\documentclass{article}
\pdfoutput=1
\usepackage{xparse}
\usepackage[utf8]{inputenc}
\usepackage{lmodern}
\usepackage{textgreek}
\usepackage{csquotes}
\usepackage{subfiles}
\usepackage[backend=biber,style=alphabetic, doi = false, isbn=false]{biblatex}
\renewbibmacro{in:}{}
\AtEveryBibitem{%
  \clearfield{primaryClass}%
  \clearfield{EDITOR}
   \clearfield{eid}
}
\usepackage{amsmath}
\usepackage{amsthm}
\usepackage{amssymb}
\usepackage{mathtools}
\usepackage{amsfonts}
\usepackage{wasysym}
\usepackage{bbm}
\usepackage{thmtools} 

\usepackage{float}

\usepackage{xargs}
\usepackage{enumitem}
\usepackage[hidelinks]{hyperref} 
\usepackage[capitalize, nameinlink]{cleveref}

\usepackage{graphicx}
\usepackage{xcolor}

\usepackage{pdfpages}

\usepackage[margin=3.4cm]{geometry}
 
\definecolor{Bnavy}{RGB}{0, 66, 128}
\definecolor{Bdust}{RGB}{140,179,217}
\definecolor{Bsugarpaper}{RGB}{198, 217, 236}
\definecolor{Bgreen}{RGB}{142, 183, 114}
\definecolor{Blimegreen}{RGB}{202, 222, 189}
\definecolor{Bgreentheme}{RGB}{36, 87, 1}

\usepackage{makecell}
\usepackage{stackengine} 

\usepackage{microtype}

\usepackage{tikz} 
\usepackage{tikz-cd}
\usetikzlibrary{ 
	calc,%
	arrows,%
	shapes,
	shapes.geometric,
    positioning,
    fit,
} 
\usetikzlibrary{backgrounds, calc, math, patterns, arrows.meta}
\usetikzlibrary{decorations.markings}

\tikzcdset{arrow style=tikz, diagrams={>=stealth}}
\usepackage[labelsep=quad,indention=10pt]{subfig}
\theoremstyle{plain}
\newtheorem{theorem}{Theorem}[subsection]

\newtheorem{lemma}[theorem]{Lemma}
\newtheorem{corollary}[theorem]{Corollary}
\newtheorem{proposition}[theorem]{Proposition}
\newtheorem{construction}[theorem]{Construction}

\theoremstyle{definition}
\newtheorem{definition}[theorem]{Definition}
\newtheorem{example}[theorem]{Example}
\newtheorem{recollection}[theorem]{Recollection}
\crefname{recollection}{Recollection}{Recollections}
\newtheorem{remark}[theorem]{Remark}
\newtheorem{notation}[theorem]{Notation}

\newtheorem{observation}[theorem]{Observation}

\newcounter{diagram}  
\crefname{diagram}{Diagram}{Diagrams}
\creflabelformat{diagram}{#2\textup{(#1)}#3}
\newenvironment{diagram}[1][]{%
    \crefalias{equation}{diagram}
    \begin{equation}%
    \begin{tikzcd}[#1]%
}{%
    \end{tikzcd}%
    \end{equation}%
}

\usepackage{notation}
 \newlist{PhiProps}{enumerate}{4}
 \setlist[PhiProps]{label*=(\roman*)}
    \crefname{PhiPropsi}{Observation}{Observations}
    \Crefname{PhiPropsi}{Observation}{Observations}
 \newlist{HProps}{enumerate}{4}
 \setlist[HProps]{label*=(\roman*)}
    \crefname{HPropsi}{Property}{Properties}
    \Crefname{HPropsi}{Property}{Properties}
 \newlist{LinkComp}{enumerate}{4}
 \setlist[LinkComp]{label*=(\roman*)}
    \crefname{LinkCompi}{Description}{Descriptions}
    \Crefname{LinkCompi}{Description}{Descriptions}
   \newlist{LocEx}{enumerate}{4}
 \setlist[LocEx]{label*=(\roman*)}
    \crefname{LocExi}{Item}{Items}
    \Crefname{LocExi}{Item}{Items}
 \newlist{LocExClaim}{enumerate}{4}
 \setlist[LocExClaim]{label*=(\alph*)}
    \crefname{LocExClaimi}{Claim}{Claims}
    \Crefname{LocExClaimi}{Claim}{Claims}
\newlist{AssTranfLem}{enumerate}{4}
 \setlist[AssTranfLem]{label*=(\arabic*)}
    \crefname{AssTranfLem}{Assumption}{Assumptions}
    \Crefname{AssTranfLemi}{Assumption}{Assumptions}
\newlist{FundamentalObs}{enumerate}{4}
 \setlist[FundamentalObs, 1]{label*=\arabic{FundamentalObsi}.}
  
    \crefname{FundamentalObs}{Observation}{Observations}
    \Crefname{FundamentalObsi}{Observation}{Observations}
    \Crefname{FundamentalObsii}{Observation}{Observations}
\newlist{approaches}{enumerate}{4}
 \setlist[approaches]{label*=Ap.\arabic*}
 \setlist[approaches,2]{label=\alph*)}
    \crefname{approaches}{Approach}{Approaches}
    \Crefname{approachesi}{Approach}{Approaches}
    \Crefname{approachesii}{Approach}{Approaches}
\newlist{motivatingQuestions}{enumerate}{4}
 \setlist[motivatingQuestions]{label*=Q.\arabic*}
    \crefname{motivatingQuestions}{Question}{Questions}
    \Crefname{motivatingQuestionsi}{Question}{Questions}
\newlist{conditionsTriang}{enumerate}{4}
 \setlist[conditionsTriang]{label*=C.\arabic*}
    \crefname{conditionsTriang}{Condition}{Conditions}
    \Crefname{conditionsTriangi}{Condition}{Conditions}
\newlist{versionHoHy}{enumerate}{4}
 \setlist[versionHoHy]{label*=V.\arabic*}
    \crefname{versionHoHy}{Version}{Versions}
    \Crefname{versionHoHyi}{Version}{Versions}   
    \title{Combinatorial models for stratified homotopy theory}
\author{Lukas Waas}
\date{January 20245}
\begin{document}
\maketitle
\begin{abstract}
    This paper is part of a series of three articles with the objective of investigating a stratified version of the homotopy hypothesis in terms of semi-model structures that interact well with classical examples of stratified spaces, such as Whitney stratified spaces. To this end, we prove the existence of several combinatorial simplicial model structures in the combinatorial setting of stratified simplicial sets. One of these we show to be Quillen equivalent to the left Bousfield localization of the Joyal model structure that presents the $(\infty,1)$-category of layered $(\infty,1)$-categories, i.e., such $(\infty,1)$-categories in which every endomorphism is an isomorphism. 
    \end{abstract}

\section{Introduction}\label{sec:comb_mod_for_strat_ho}
Stratified spaces were originally introduced by Whitney, Thom and Mather as a tool to investigate spaces with singularities (see \cite{Whitney,mather1970notes,mather1973strat,thom1969ensembles}).
In the broadest sense, a stratified space consists of the data of a topological space together with a decomposition into disjoint pieces, with additional varying assumptions on the properties of these pieces - the so-called strata - and their interactions.
In more recent years, the investigation of such objects has shifted from being primarily concerned with studying a single object to studying classes of stratified spaces and the stratified maps between them (such maps that map strata into strata). 
Even more, instead of focusing on this 1-categorical perspective, the focus has been on the $(\infty,1)$-categorical point of view: Studying homotopy theories of (certain classes of) stratified spaces, induced by stratified notions of homotopy (see \cite{quinn1988homotopically,HughesPathSpaces,miller2013,LocalStructOnStrat,AFRStratifiedHomotopyHypothesis,douteauEnTop,douteauwaas2021,haine2018homotopy,nand2019simplicial}). Starting with Quinn's theory of homotopically stratified spaces (named homotopically stratified sets in \cite{quinn1988homotopically}), several homotopy theories of stratified spaces were introduced and studied, for example, in \cite{douteauwaas2021,haine2018homotopy,nand2019simplicial}.
This paper is part of a three-part series of articles concerned with these homotopy theories of stratified spaces, the goal of which is to develop (semi-)model structures for stratified homotopy theory which interact well with classical geometric and topological examples of stratified spaces, and ultimately lead to a tractable and interpretable version of the so-called \textit{topological stratified homotopy hypothesis}: 
\begin{displayquote}{The homotopy theory of stratified topological spaces is the same as the homotopy theory of such $(\infty,1)$-categories in which every endomorphism is an isomorphism.} 
\end{displayquote}
It is a general paradigm in homotopy theory (most prominently realized in \cite{Quillen}) that homotopy theoretic phenomena are often easier understood after being translated into the world of combinatorics. Thus, our approach to constructing model structures for stratified homotopy theory consists of developing the theory in a combinatorial framework first and then transferring it to the world of stratified topological spaces.\\
The goal of this paper is to cover the purely combinatorial part of this program. To this end, we survey several model structures for stratified simplicial sets over a fixed poset already on the market, exposing the precise connections between them. 
We then extend these model categories to model categories of stratified simplicial sets with varying poset, and connect the latter to the Joyal model category for $(\infty,1)$-categories.
Let us explicitly state that our goal here is not to obtain results which are new from a purely conceptual ($\infty$-categorical) point of view, but rather to produce an overview of combinatorial models for stratified homotopy theory and mirror several results and structures already known on the $\infty$-categorical level from \cite{Exodromy,haine2018homotopy} in the language of model categories. This has the advantage that it will ultimately allow us to transfer these structures and results to the topological stratified framework, in which the additional structure of a (semi)model category is necessary to connect the homotopy theory with the geometry and topology of stratified spaces. \\
\\
In more detail, the content of this paper is as follows. First, in \cref{subsec:recol_dou_haine}, we recall the Douteau-Henriques model structure (defined by Douteau in \cite{douSimp}, and independently defined by Henriques in \cite{Henriques}), as well as the Joyal-Kan model structure defined by Haine in \cite{haine2018homotopy}, which are both defined on categories of simplicial sets stratified over a fixed poset. The latter of these presents $\infty$-categories with a conservative functor into a poset, so-called \textit{abstract stratified homotopy types}. We show that the Joyal-Kan model structure is the left Bousfield localization of the Douteau-Henriques model structure at the class of inner stratified horn inclusions (\cref{prop:loc}). This provides a useful approach to investigating the categorical homotopy theories of stratified spaces defined by Haine, and the one defined by Nand-Lal in \cite{nand2019simplicial}: One can often obtain results about the categorical theories from results about the Douteau-Henriques theories, which often turn out to be significantly easier to handle, due to the explicit description of weak equivalences in the latter (see, for example, the proof of \cite[Thm 0.1.1]{haine2018homotopy}).
To illustrate this method, in \cref{subsec:decollages}, we provide combinatorial simplicial model structure for the homotopy theory of d\'ecollages described in \cite{haine2018homotopy} - roughly space valued presheaves indexed over the finite increasing sequences over a poset fulfilling a Segal style fibrancy condition -  and prove a Quillen equivalence between Haine's model structure for abstract stratified homotopy types and the model structure for d\'ecollages (\cref{prop:equ_decol_haine}). This Quillen equivalence presents an equivalence of $\infty$-categories already proven in \cite{haine2018homotopy}, without appealing to the theory of complete Segal spaces. Our proof works by constructing a new left Quillen functor model for the functors of homotopy links studied in detail in \cite{douteauwaas2021} (\cref{con:geometric_simplicial_link,theo:geometric_link}). \\
Then, in \cref{subsec:from_local_to_global} we move from the case of a fixed poset to the case of flexible posets by gluing the model structures described in \cite{douSimp,haine2018homotopy} using a method of \cite{CagneMellies}, already employed in \cite{douteauEnTop}. These model structures provide combinatorial simplicial models for the homotopy theories of stratified spaces with varying posets investigated in \cite{douteauEnTop,haine2018homotopy} (see \cref{prop:sStraC_pres_Astrat}).\\
Both of the homotopy theories of stratified simplicial sets constructed in \cref{subsec:from_local_to_global} have the property that morphisms in them are not entirely determined by the underlying map of spaces but include the additional data of a map of posets, in opposition to the classical scenario (see, for example, \cite{weinberger1994topological}) where stratification was purely a property of a map (see the beginning of \cref{subsec:refining_strat_sset} and particularly \cref{rem:so_many_empty_spaces}). If one is looking to get closer to the classical scenario, one can instead work with so-called \textit{refined stratified simplicial sets} (called $0$-connected in the case of abstract stratified homotopy types in \cite{Exodromy}), which are, roughly speaking, the class of stratified simplicial sets for which the underlying poset is entirely encoded in the closure relations of the strata (see \cref{def:refined}).
To account for this, in \cref{subsec:refining_strat_sset}, we provide right Bousfield localizations of the global model structures in which maps between bifibrant objects have the property that maps are defined entirely on the space (simplicial set) level. Our main results in this subsection may be summarized as follows.
The category of stratified simplicial sets $\sStratN$ admits the structures of two combinatorial simplicial model categories, $\sStratDRN$ and $\sStratCRN$ which are respectively right Bousfield localizations of the global versions of the Douteau-Henriques and the Joyal-Kan model structures (\cref{prop:ex_red_struct}). In both model structures, the cofibrant objects are precisely the refined stratified simplicial sets.
$\sStratCRN$ is the left Bousfield localization of $\sStratDRN$ at inner stratified horn inclusions and presents the $\infty$-category of refined ($0$-connected) abstract stratified homotopy types (\cref{prop:c_is_left_bous_of_d_ref,prop:sStraC_pres_Astrat}). We furthermore show that weak equivalences in both $\sStratCRN$ and $\sStratDRN$ are stable under filtered colimits, which is one of the key ingredients to transferring these model structures to the topological realm (\cref{prop:strat_we_stable_colim}).\\
 In the next subsection (\cref{subsec:from_ref_to_ordered_qc}), we then show that one of these model structures is Quillen equivalent to the left Bousfield localization of the Joyal model structure on simplicial sets that presents the $\infty$-category of small $\infty$-categories in which every endomorphism is an isomorphism (\cref{prop:Quillen_Equ_betw_ref_and_ord}).
This lifts a result proven on the $\infty$-categorical level in \cite[2.3.8]{Exodromy} to the level of model categories, and provides one necessary core result for our version of the topological stratified homotopy hypothesis proven in \cite{TSHHWa}. Finally, we prove that all of the model structures on $\sStratN$ defined in this paper are cartesian closed (\cref{thm:cartesian_closure}), which allows us to recover a result of Bruce Hughes (see \cite[Main Result]{HughesPathSpaces}) on homotopically stratified spaces in our purely combinatorial setting (\cref{con:global_exit_paths}).
\subsection{Language and notation}\label{subsec:simp_lang_not}
Let us begin by introducing some of the relevant categories and recalling some notation. We will follow the convention of denoting $1$-categories in bold letters, simplicial categories in bold underlined letters, and $(\infty,1)$-categories (modeled by quasi-categories) by writing their first capital letter in caligraphic script. If we wish to denote the underlying $1$-category of a simplicial category, we do so by simply omitting the line under the name. We use the same notation for model categories mutatis mutandis.
\begin{notation} We are going to use the following terminology and notation for partially ordered sets, drawn partially from \cite{douSimp} and \cite{haine2018homotopy}:
\begin{itemize}
    \item  We denote by $\Pos$ the category of partially ordered sets, with morphisms given by order-preserving maps.
    \item  We denote by $\Delta$ the full subcategory of $\Pos$ given by the finite, nonempty, linearly ordered posets of the form $[n]:=\{0, \cdots, n\}$, for $n \in \mathbb N$.
    \item  Given $P \in \Pos$, we denote by $\catFlag$ the slice category $\Delta_{/ \pos}$. 
    That is, objects are given by arrows $[n] \to \pos$ in $\Pos$, $n \in \mathbb N$, and morphisms are given by commutative triangles.
    \item We denote by $\catRFlag$ the \define{subdivision of } $\pos$, given by the full subcategory of $\catFlag$ of such arrows $[n] \to \pos$, which are injective.
    \item The objects of $\catFlag$ are called \define{flags of} $\pos$. We represent them by strings $[p_0 \leq \cdots \leq p_n]$, of $p_i \in \pos$. We refer to $n$ as the \define{length} of the flag $[p_0 \leq \cdots \leq p_n]$.
    \item Objects of $\catRFlag$ are called regular \define{flags of} $\pos$. We represent them by strings $\standardFlag$, of $p_i \in \pos$.
\end{itemize}
\end{notation}
\begin{notation}
    We use the following terminology and notation for (stratified) simplicial sets, drawn partially from \cite{douSimp} and \cite{haine2018homotopy}:
    \begin{itemize}
    \item We denote by $\sSet$ the simplicial category of simplicial sets, i.e. the category of set valued presheaves on $\Delta^{\op}$, equipped with the canonical simplicial structure induced by the product (see \cite{HigherTopos} for all of the standard notation used for simplicial sets).
    \item When we treat $\sSet$ as a model category this will generally be with respect to the Kan-Quillen model structure (see \cite{Quillen}), unless otherwise noted. When we use Joyal's model structure for quasi-categories (\cite{joyalNotes}) instead, we will denote this model category by $\sSetJoy$. 
    \item We think of $\Pos$ as fully faithfully embedded in $\sSetN$, via the nerve functor (compare \cite{haine2018homotopy}). By abuse of notation, we just write $\pos$ for the simplicial set given by the nerve of $\pos \in \Pos$. 
    \item For $\pos \in \Pos$, we denote by $\sSetPN$ the slice category $\sSetN_{/\pos}$, which is equivalently given by the category of set-valued presheaves on $\catFlag$. We treat $\sSetP$ as a simplicial category, denoted $\sSetP$, with the structure inherited from $\sSet$ (see \cite[Recol. 2.21.]{douteauwaas2021}).
    \item Objects of $\sSetP$ are called \define{$\pos$-stratified simplicial sets}. They are given by a tuple $\str =(\ustr,\sstr \colon \ustr \to \pos)$. In the literature, a $P$-stratified simplicial set $\str =(\ustr,\sstr \colon \ustr \to \pos)$ is often simply referred to by its underlying simplicial set $\ustr$, omitting the so-called \define{stratification} $\sstr \colon \ustr \to \pos$. We are not going to adopt this notation here, as we will frequently consider the same simplicial set with changing stratifications. We are always going to use calligraphic letters for stratified simplicial sets and their non-calligraphic counterparts for the underlying simplicial set.
    \item  Morphisms in $\sSetP$ are called \define{stratum-preserving simplicial maps.} 
    Simplicial homotopies in $\sSetP$ are called \define{stratified simplicial homotopies.}
    Simplicial homotopy equivalences in $\sSetP$ are called \define{stratum-preserving simplicial homotopy equivalences}.
   \item Given a map of posets $f\colon Q \to \pos$ and $\str \in \sSetP$, we denote by $f^*\str \in \sSetP[Q]$ the stratified simplicial set $\ustr \times_{P} \pos[Q] \to Q$. We are mostly concerned with the case where $f$ is given by the inclusion of a singleton $\{p \}$, of a subset $\{q \sim p \mid q \in \pos \}$, for $p \in \pos$ and $\sim$ some relation on the partially ordered set $\pos$ (such as $\leq$), or more generally, a subposet $Q \subset \pos$. We then write $\str_{p}$ (or, respectively, $\str_{\sim p}$, $\str_Q$) instead of $f^*\str$. The simplicial sets $\str_{p}$, for $p \in \pos$ are called the \define{strata of} $\str$. 
   \item For $f\colon Q \to \pos$ in $\Pos$, we denote by $f_!$ the left adjoint to the simplicial functor $f^* \colon \sSetP[P] \to \sSetP[Q]$, given on objects by $(\sstr \colon \ustr  \to Q) \mapsto (f \circ \sstr \colon \ustr \to Q \to \pos)$.
    \item Let $\sSetN^{[1]}$ be the category of arrows of simplicial sets.
    We denote by $\sStratN$ the category of all stratified simplicial sets, given by the full sub-category of $\sSetN^{[1]}$ of such arrows $X \to \pos$, where $X \in \sSetN$ and $\pos\in \Pos$ is (the nerve of) a poset. In particular, every object of $\sStratN$ is given by a $\pos$-stratified simplicial set, for some $\pos\in \Pos$, and a morphism $(X \to \pos) \to (Y \to \pos[Q]$) is given by a pair of morphisms $f\colon X \to Y$ and $g \colon \pos \to \pos[Q]$, where $f$ is a simplicial map and $g$ can be seen as a map of posets, 
    making the obvious square commute (see also \cite[Def. 2.19]{douteauwaas2021}).
    Morphisms in $\sStratN$ are called \define{stratified simplicial maps}. 
    \item Given $\str \in \sStratN$, we are going to use the notational convention $\str = (\utstr, \sstr, \pstr)$ to refer, respectively, to the underlying simplicial set, the stratification and the poset and proceed analogously for morphisms. 
    \item We equip $\sStratN$ with the structure of a simplicial category, tensored and cotensored over $\sSet$, denoted $\sStrat$, with the tensoring induced by setting \[\str \otimes \Delta^n = (\ustr \times \Delta^n  \to \ustr \to \pstr) \spaceperiod\]
    Simplicial homotopies in $\sStrat$ are called \define{stratified simplicial homotopies.}
    Simplicial homotopy equivalences in $\sStrat$ are called \define{stratified simplicial homotopy equivalences}.
    \item The forgetful functor $\sStratN \to \sSetN$, $\str \mapsto \ustr$ will be denoted $\forget$, and has a right adjoint and a left adjoint. The left adjoint is given by left Kan extending the functor on simplices: $\Delta^n \mapsto \{ \Delta^n \xrightarrow{1_{\Delta^n}} \Delta^n = [n] \}$. We denote it by $\lstr \colon \sSetN \to \sStratN$. The right adjoint is given by mapping $K \in \sSetN$ to the trivially stratified simplicial set $\{ K \to [0] \}$. By abuse of notation, we will often write $K$ to refer to the trivially stratified simplicial set associated to a simplicial set $K$.
    
    \end{itemize}
    \end{notation}
    \begin{remark}
        There is a canonical forgetful functor, $\sStratN \to \Pos$ given by $\str \mapsto \pstr$ and we may identify its fiber at $\pos \in \Pos$ with $\sSetPN$. This functor is easily seen to be a Grothendieck bifibration, with right action given by $f \mapsto f^*$ and left action given by $f \mapsto f_!$. It follows that we may use the results in \cite{CagneMellies} to glue local model structures on the fibers to global model structures.
    \end{remark}
     \begin{remark}\label{rem:sstrat_bicomplete}
        Both $\sStratN$ and $\sSetPN$, for $P \in \Pos$, are bicomplete categories (see, for example, \cite{douSimp}). Limits and colimits in $\sSetPN$ are simply given by the limits and colimits in a slice category. Both limits and colimits in $\sStratN$ are computed by taking , respectively, the limit or colimit both on the simplicial set and on the poset level.
    \end{remark}
    \begin{notation}\label{not:flag_operations}
    We are going to need some additional notation for flags and stratified simplices.
    \begin{itemize}
    \item For a flag $\J = [p_0 \leq \cdots \leq p_n] \in \Delta_P$, we write $\Delta^\J$ for the image of $\J$ in $\sSetPN$ under the Yoneda embedding $\catFlag \hookrightarrow \sSetPN$. Equivalently, $\Delta^\J$ is given by the unique simplicial map $\Delta^{n} \to \pos$ mapping $i \mapsto p_i$. $\Delta^\J$ is called the \define{stratified simplex} associated to $\J$.
     \item Given a stratified simplex $\Delta^\J$, for $\J = [p_0 \leq \cdots \leq p_n]$, we write $\partial \Delta^\J$ for its \define{stratified boundary}, given by the composition $\partial \Delta^n \to \Delta^n \to \pos$.
    \item Furthermore, for $0 \leq k \leq n$, we write $\Lambda^\J_k \subset \Delta^\J$ for the stratified subsimplicial set given by the composition $\Lambda^n_k \to \Delta^n \to \pos$ (we use the horn notation as in \cite{HigherTopos}). The stratum-preserving map $\Lambda^\J_k \hookrightarrow \Delta^\I$ is called the \define{stratified horn inclusion associated to } $\J$ and $k$. The inclusion $\Lambda^\J_k \hookrightarrow \Delta^\I$ is called \define{admissible}, if $p_k = p_{k+1}$ or $p_{k} = p_{k-1}$. The inclusion $\Lambda^\J_k \hookrightarrow \Delta^\I$ is called \define{inner} if $0 <k <n$.
    \item Using the fully faithful (and continuous) embedding $\catFlag \hookrightarrow \sSetPN$, we extend the base-change notation for stratified simplicial sets to flags.  That is, for $f \colon Q \to \pos$ we write $f^*\J$ for the unique flag of $Q$ corresponding to $f^*(\Delta^\J)$. We use the same shorthand notation for subsets $Q \subset \pos$. For example, $\J_{\leq p}$ is the flag obtained from $\J$, by removing all entries not less than or equal to $p$.
    \item It will also be convenient to have a concise notation for the images of simplices, horns, and boundaries under $\lstr \colon \sSetN \to \sStratN$. These are denoted by replacing the exponent $n \in \mathbb N$, by the poset $[n]$. That is, we write $ \stratSim := \lstr(\Delta^n)$, $ \stratBound := \lstr(\partial \Delta^n)$, $\stratHorn := \lstr(\Lambda^n_k)$, for $0 \leq k \leq n$. 
\end{itemize}
\end{notation}

\section{Combinatorial models over a fixed poset}
Before we begin with the construction of model structures for the category of stratified simplicial sets over varying posets $\sStratN$, we first cover the case of categories of stratum-preserving maps. Later, in \cref{sec:global_mod_struct}, we will piece together the model structures defined in this section for one fixed poset, to obtain model structures on $\sStratN$.
For the remainder of this subsection, fix some poset $\pos$.
\subsection{ The minimalist- and the Joyal-Kan approach}\label{subsec:recol_dou_haine}
In this subsection, we recall the model structures on $P$-stratified simplicial sets defined in \cite{douSimp} and \cite{haine2018homotopy} and point out the precise relationship between them. 
Since $\sSetPN$ is isomorphic to the category of set-valued presheaves on $\Delta_P$, we may use the methods of \cite{CisinskiPrefaisceaux} to construct model structures on it. 
\begin{recollection}[\cite{douteauwaas2021}]\label{rec:Dou-Hen-mod}
The \define{Douteau-Henriques model structure} on $\sSetP$, defined first in \cite{douSimp}, is the Cisinski model structure (see \cite[Thm. 2.4.19]{HigherCatCisinki}) induced by the simplicial cylinder $X \mapsto X \otimes \Delta^1$, with the empty set of anodyne extensions. This defines a combinatorial, cofibrant, simplicial model structure on $\sSetP$ whose defining classes may be characterized as follows (see \cite{douteauwaas2021}) for this characterization, which is stronger than the one provided in \cite{douSimp}):
\begin{enumerate}
    \item Cofibrations are precisely the monomorphisms in $\sSetP$.
    \item Weak equivalences are precisely such stratum-preserving simplicial maps $\str \to \str[Y]$ for which the induced map of simplicial sets
    \[
    \sSetP(\Delta^\I, \str) \to \sSetP( \Delta^\I, \str[Y])
    \]
    is a weak homotopy equivalence, for all $\I \in \sd(\pos)$. We call such a map a \define{diagrammatic equivalence}.
    \item Fibrations are precisely the simplicial maps which have the right lifting property with respect to all admissible horn inclusions.
\end{enumerate}
We denote the resulting simplicial model category by $\sSetPD$. It carries the minimal model structure (with respect to weak equivalences) in which the cofibrations are the monomorphisms, and stratified simplicial homotopy equivalences are weak equivalences. $\sSetPD$ is cofibrantly generated by the classes of stratified boundary inclusions and admissible horn inclusions.
\end{recollection}
Since $\sSetPD$ is in some sense minimal among model structures on $\sSetP$, it is not surprising that alternative theories arise as a localization of the homotopy theory presented by $\sSetPD$. In particular, this is the case for the model structure defined in \cite{haine2018homotopy}.
\begin{recollection}[\cite{haine2018homotopy}]\label{recol:haine_mod_cat}
    The \define{Joyal-Kan model structure} on $\sSetP$ is the one obtained by localizing the model structure inherited from the Joyal-model structure on $\sSetN$ at the cylinder $- \otimes \Delta^1$. The Joyal-Kan model structure is simplicial, cofibrantly generated, and its defining classes have the following descriptions:
    \begin{enumerate}
        \item Cofibrations are precisely the monomorphisms in $\sSetPN$.
        \item Fibrant objects are precisely the stratified simplicial sets $\str$ for which the underlying simplicial set $\ustr$ is a quasi-category and $\sstr\colon \ustr  \to \pstr$ is a conservative functor. Fibrations between fibrant objects are precisely the stratum-preserving simplicial maps that have the right lifting property with respect to all inner and admissible stratified horn inclusions.
        \item Weak equivalences between fibrant objects are equivalently characterized as the class of 
        \begin{enumerate}
            \item stratified homotopy equivalences;
            \item Joyal equivalences (over $\pos$);
            \item stratum-preserving maps that induce weak equivalences on $\sSetP( \Delta^\I,-)$, for all regular flags $\I$ of length lesser or equal to $1$. 
        \end{enumerate}
    \end{enumerate}
    We denote the model category (uniquely determined by these classes) by $\sSetPC$. Weak equivalences in this model structure will be called \define{Joyal-Kan equivalences}. It follows by the characterization of weak equivalences and fibrant objects above that $\sSetPC$ presents the $\infty$-category of conservative functors from a quasi-category into $\pos$, also called \define{abstract stratified homotopy types} over $\pos$.
\end{recollection}
In view of the minimality of $\sSetPD$, the following is not surprising:
\begin{proposition}\label{prop:loc}
    The simplicial model category $\sSetPC$ is the left Bousfield localization of $\sSetPD$ at the class of stratified inner horn inclusions.
\end{proposition}
\begin{proof}
    It suffices to see that the localization described above has the same fibrant objects as $\sSetPC$. Let $\str$ be fibrant in the localization. In particular, $\str$ has the filler property for all admissible and inner stratified horn inclusions. It follows that $\ustr$ is a quasi-category. Furthermore, as $\ustr \to \pos$ has the right lifting property with respect to every admissible horn inclusion (which includes horn inclusions entirely contained in one stratum) for every $p \in \pos$ the stratum $\str_p$ is a Kan complex. In particular,  $\sstr \colon \ustr \to \pos$ is conservative. Now, conversely, suppose that $\str$ is such that $\sstr \colon \ustr \to \pstr$ is a conservative functor of quasi-categories. Then, since $P$ is the nerve of a $1$-category, $\sstr$ is also an inner fibration, which shows that $\str$ admits fillers for all inner horn inclusions. 
    Now, consider a horn inclusion $\Lambda^\J_k \hookrightarrow \Delta^\J$, with $\J = [p_0 \leq \dots \leq p_n]$, which is not inner, but admissible. We cover the case $k = n$, as the other is analogous. Hence, we may assume that $p_{n-1} = p_n$. Since $\sstr$ is conservative, it follows that the edge of $\Lambda^n_k$ from $n-1$ to $n$ maps to an isomorphism $f$ in $\ustr$. In particular, $f$ is cartesian (\cite[Prop. 2.4.1.5]{HigherTopos}) and a lift with respect to $\Lambda_k^\I \hookrightarrow \Delta^\J$ exists by \cite[Rem. 2.4.1.4]{HigherTopos}.  Therefore, $\str$ admits a filler for all inner and all admissible horn inclusions. The latter shows that it is fibrant in $\sSetPD$. To see that $\str$ is local with respect to inner horn inclusions $\Lambda^\I_k \hookrightarrow \Delta^\I$, we may equivalently show that $\str \to \pos$ has the right lifting property with respect to the maps
    \[
    \Lambda^\I_k \otimes  \Delta^n \cup_{\Lambda^\I_k \otimes \partial \Delta^n} \Delta^\I \otimes \partial \Delta^n \hookrightarrow \Delta^\I \otimes \Delta^n.
    \]
    It is a standard argument that these may be decomposed into a composition of pushouts of inner horn inclusions (see, for example, \cite[Cor 3.2.4]{HigherCatCisinki}).
\end{proof}
Again, using \cite[Cor 3.2.4]{HigherCatCisinki}, we obtain:
\begin{corollary}
    Fibrant objects in $\sSetPC$ are precisely such stratified simplicial sets that have the horn filling property with respect to all admissible and inner stratified horn inclusions.
\end{corollary}
\cref{prop:loc} is particularly useful, because it provides a criterion to check for weak equivalences $\sSetPC$. Generally, in $\sSetPC$, the lack of an explicit criterion for weak equivalences can make such verifications challenging. In many cases, however, we may already verify the relevant property in $\sSetPD$ and then use the fact that they are preserved under left Bousfield localization.\footnote{See for example the approach to proving the existence of semi-models tructures on the topological side we take in \cite{TSHHWa}. Similarly, the proof of a version of a stratified homotopy hypothesis in \cite{haine2018homotopy} was built on \cite{douteauEnTop}.} In this sense, the homotopy theories defined by $\sSetPC$ and $\sSetPD$ are really not in a competing, but in a mutually supportive relationship. 
Let us finish this subsection with a general remark and a proposition which we use to transfer model structures from the simplicial to the topological world in \cite{TSHHWa}.
\begin{remark}\label{rem:hocolim_of_simp}
    Every stratified simplicial set $\str \in \sSetPDN$ (or in $\sSetPCN$) is the homotopy colimit of its stratified simplices. In fact, by \cite[Ex. 8.2.5; Prop. 8.2.9]{CisinskiPrefaisceaux}, this holds for any Cisinski model structure on $\sSetPN$.
\end{remark}
\begin{proposition}\label{prop:we_stable_colim}
    Weak equivalences in $\sSetPD$ and $\sSetPC$ are stable under filtered colimits. 
\end{proposition}
\begin{proof}
    For $\sSetPC$ this is \cite[2.5.9]{haine2018homotopy}. For $\sSetPD$ this follows from the fact that weak equivalences are detected by a finite set of functors with values in simplicial sets, that preserve filtered colimits. That weak equivalences of simplicial sets are stable under filtered colimits follows, for example, as an application of Kan's $\textnormal{Ex}^{\infty}$ functor, which preserves all filtered colimits (see \cite{KanExInf}).
\end{proof}
\subsection{Homotopy links and a model structure for d\'ecollages}\label{subsec:decollages}
    As is apparent from the characterization of weak equivalences in $\sSetPD$ (\cref{rec:Dou-Hen-mod}), the simplicial mapping spaces $\sSetP(\Delta^\I,\str)$, for $\str \in \sSetP$, play a central role in understanding the homotopy theory of $\sSetPD$.
\begin{recollection}\label{recol:equivalence_inj_mod}
     For $\I \in \sd(\pos)$ and $\str \in \sSetP$, the simplicial set $\sSetP(\Delta^\I,\str)$ is called the \define{$\I$-th (simplicial) homotopy link of $\str$}. It is denoted $\HolIPS(\str)$ (see also \cite[Def. 2.31]{douteauwaas2021}). The simplicial sets $\HolIPS(\str)$ are organized in the structure of a simplicial presheaf on $\sd(\pos)$, denoted $\HolIPS[](\str)$. Denote by $\SDiag$ the simplicial category of simplicial presheaves on $\sd P$.
     Homotopy links induce a nerve-style functor 
        \begin{align*}
            \HolIPS[] \colon \sSetP \to \SDiag
        \end{align*}
    that admits a left adjoint, given by mapping $D \in \SDiag$ to the coend $\int^{\I} \Delta^\I \otimes D_{\I}$. 
    This left adjoint functor preserves all monomorphisms. Furthermore, it preserves all weak equivalences in both directions, by \cite[Thm. 1.3]{douteauwaas2021}. Hence, we obtain a pair of (simplicial) Quillen adjoint functors
    \[
     \int^{\I} \Delta^\I \otimes -_{\I}:\DiagSIn \rightleftharpoons  \sSetPD \colon \HolIPS[]
    \]
     between $\sSetPD$ and $\SDiag$ equipped with the injective model structure.
\end{recollection}
As an immediate corollary of \cite[Thm. 1.3]{douteauwaas2021}, one obtains the following. Recall that a functor $F$ between categories with weak equivalences is said to \textit{create weak equivalences}, if it has the property that $F(w)$ is a weak equivalence, if and only if $w$ is a weak equivalence, for every morphism $w$ in the source category.
\begin{corollary}\label{cor:equ_inj_mod_struct}
    The simplicial Quillen adjunction
    \[
     \int^{\I} \Delta^\I \otimes -_{\I}:\DiagSIn \rightleftharpoons  \sSetPD \colon \HolIPS[]
    \]
    is a Quillen equivalence that creates weak equivalences in both directions.
\end{corollary}
\begin{remark}\label{rem:equ_char_of_we}
    Note that the condition for an adjunction between model categories (more generally, categories with weak equivalences) to create weak equivalences in both directions is equivalent to both functors preserving weak equivalences and the unit and counit being given by weak equivalences.
\end{remark}
We may interpret \cref{cor:equ_inj_mod_struct} as follows. If one takes the perspective that inclusions of stratified simplicial sets should be the cofibrations, and that at least the stratified (simplicial) homotopy equivalences should be weak equivalences, then the minimal homotopy theory one ends up with is the one of simplicial presheaves on $\sd(\pos)$. 
Consequently, one would also expect to be able to interpret the homotopy theory of $\sSetPC$ in terms of a category of (certain) presheaves on $\sd(\pos)$. Such a result was first shown in \cite[Thm. 2.7.4]{Exodromy} and in \cite{haine2018homotopy}. Here, we are going to give a version of this result in the language of model categories. This serves to illustrate a method of proof, which we are also going to employ when we show the existence of convenient model structures for topological stratified spaces in \cite{TSHHWa}:\\
Up to weak equivalence, the functor $\HolIPS[]: \sSetP \to \SDiag$ is both a right and a left Quillen adjoint.
       \begin{construction}\label{con:geometric_simplicial_link}
        Homotopy links admit a more geometric model, which is constructed as follows. Let $\I = \standardFlag$ be a regular flag of $P$. We then obtain a functor 
        \[
        \LinkI \colon \Delta_P \to \sSetN
        \] by mapping
        \[
        \J \mapsto \prod_{p_i \in \I} \Delta^{\J_{p_i}}.
        \]
        If $\I_0 \subset \I_1$, then the projections of the product induce a natural transformation 
        \[ \LinkI[\I_1] \to \LinkI[\I_0],\] Under left Kan extension, we therefore obtain a functor 
        \[
        \LinkI[] \colon \sSetPN \to \Diag.
        \]
        \end{construction}
Let us explicitly compute $\LinkI$ for a stratified version of Joyal's join functor.
\begin{construction}\label{con:stratified_join}
    Suppose that $\I \in \sd(\pos)$ is a non-degenerate flag such that $\I = \I_0 \cup \I_1$, with $\I_0$ and $\I_1$ disjoint (and non-empty). 
    Given two flags $\J_0$ an $\J_1$ degenerating from a subflag of $\I_0$ and $\I_1$ respectively, the associated object $\J_0,\J_1 \in \Delta_{\pos}$ admits a coproduct, denoted $\J_0 \sqcup \J_1$. 
    It is given by the (appropriately ordered) union of the sequences defining $\J_0$ and $\J_1$.
    In particular, whenever $\J$ degenerates from a subflag of $\I$ that intersects $\I_0$ and $\I_1$ non-trivially, then $\J = \J_0 \sqcup \J_1$, where $\J_0$ and $\J_1$ denote the respective restrictions of $\J$ to $\I_0$ and $\I_1$.
    Let $\str \in \sSetPN[\I_0]$ and $\str[Y] \in \sSetPN[\I_1]$. We denote by $\str *_P \str[Y] $ the stratified simplicial set given by the presheaf on $\Delta_P$ mapping
    \[
    \J \mapsto \begin{cases}
        \emptyset & \textnormal{, if $\J$ does not degenerate from a subflag of $\I$}\\
        \str({\J}) &\textnormal{, if $\J$ degenerates from a subflag of $\I_0$} \\
        \str[Y]({\J}) &\textnormal{, if $\J$ degenerates from a subflag of $\I_1$} \\
        \str(\J_0) \times  \str[Y](\J_1)& \textnormal{, if $\J = \J_0 \sqcup \J_1$ for $\J_0$ and $\J_1$ as above}
    \end{cases}
    \]
    with all face and degeneracy maps induced by the ones on $\str$ and $\str[Y]$, the functoriality of restriction to $\I_0$ and $\I_1$ and the universal property of the product.
    This construction induces a functor
    \[
    - *_P - \colon \sSetPN[\I_0] \times \sSetPN[\I_1] \to \sSetPN,
    \]
    functorial in morphisms in the obvious way.
    It comes together with a natural transformation
    \[
    \str \sqcup \str[Y] \hookrightarrow \str *_P \str[Y].
    \]
    where we treat $\str \sqcup \str[Y]$ as a stratified simplicial set over $\pos$.
    We call this construction the \define{$P$-stratified join functor.} Indeed, if we restrict to stratified simplices, then there is a canonical natural isomorphism
    \[
    \Delta^{\J_0} *_P \Delta^{\J_1} \cong \Delta^{\J_0 \sqcup \J_1}
    \]
    If we fix any of the two arguments (say the first, which suffices, since the construction is symmetric), then we may use this natural transformation to obtain a lift of the stratified join functor.
    \begin{align*}
      \str[X] *_P - \colon \sSetPN[\I_1] &\to (\sSetPN)_{\str[X]/} \\
      \str[Y] &\mapsto (\str[X] \hookrightarrow \str \sqcup \str[Y] \hookrightarrow \str *_P \str[Y]).
    \end{align*}
    It follows immediately from the definition of 
    $\str[X] *_P -$ and the elementary laws for computing colimits in presheaf and under-categories that $ \str[X] *_P -$ is cocontinuous as a functor with image in $(\sSetPN)_{\str[X]/}$. \\
    \\
    \end{construction}
    Let us now take a look at the interaction of the stratified join functor with the link functors.
    \begin{lemma}
        Using the notation of \cref{con:stratified_join}, there is a natural isomorphism of bifunctors
        \[
        \LinkI (- *_P -) \cong \LinkI[\I_0](-) \times  \LinkI[\I_1](-).
        \]
    \end{lemma}
    \begin{proof}
        We use the notation of \cref{con:stratified_join}.
        Observe that that there is a canonical isomorphism \[
        \LinkI ( \Delta^{\J_0 \sqcup \J_1}) \cong \LinkI[\I_0](\Delta^{\J_0}) \times \LinkI[\I_1](\Delta^{\J_1}) \spaceperiod \] 
        We have already seen that there is an isomorphism
        \[
        \Delta^{\J_0} *_{\pos}\Delta^{\J_0} \cong  \Delta^{\J_0 \sqcup \J_1}
        \]
        natural in $\J_0$ and $\J_1$.
        Hence, after restricting to $\Delta_{\I_0} \times \Delta_{\I_1}$, there is a natural isomorphism of bivariate functors 
        \[
        \LinkI (- *_P -) \cong \LinkI[\I_0] \times  \LinkI[\I_1].
        \]
        We now want to extend this isomorphism to a natural isomorphism of functors on all of $\sSetPN[\I_0] \times \sSetPN[\I_1]$. To see this, via left Kan extension in both arguments, it suffices to show that $\LinkI (- *_P -)$ is cocontinuous in both arguments. Note that $\str[X] *_P -$ is only cocontinuous as a functor into the under-category, hence an additional argument is required.
        To this end, observe that colimits in the under category $(\sSetPN)_{\str[X]/}$ of a diagram of arrows $ i\mapsto (\str \xrightarrow{f_i} \str[Y]_i$) can be computed as the lower horizontal arrow in the following pushout
\begin{diagram}
	{\colim \str[X]} & {\colim \str[Y]_i} \\
	{\str[X]} & {\str[Z]} \spaceperiod
	\arrow["{\colim f_i}", from=1-1, to=1-2]
	\arrow[from=1-1, to=2-1]
	\arrow[from=1-2, to=2-2]
	\arrow[from=2-1, to=2-2]
	\arrow["\lrcorner"{anchor=center, pos=0.125, rotate=180}, draw=none, from=2-2, to=1-1]
\end{diagram}
        In particular, given a colimit of a diagram of stratified $i \mapsto \str[Y]_i \in \sSetPN[\I_1]$ and $\str[X] \in \sSetPN[\I_0]$ there is a pushout square
        \begin{diagram}
	{\colim \str[X]} & {\colim (\str[X] *_P \str[Y]_i}) \\
	{\str[X]} & {\str[X] *_P \colim \str[Y]_i} \spaceperiod
	\arrow["{\colim f_i}", from=1-1, to=1-2]
	\arrow[from=1-1, to=2-1]
	\arrow[from=1-2, to=2-2]
	\arrow[from=2-1, to=2-2]
	\arrow["\lrcorner"{anchor=center, pos=0.125, rotate=180}, draw=none, from=2-2, to=1-1]
\end{diagram}
    in $\sSetPN$.
        If we apply the colimit preserving functor $\LinkI$ to this square, we obtain a pushout square of simplicial sets
        \begin{diagram}
	\colim \LinkI (\str[X]) & \colim\LinkI ( \str[X] *_P \str[Y]_i) \\
	\LinkI ({\str[X]}) & \LinkI ({\str[X] *_P \colim \str[Y]_i}) \spaceperiod
	\arrow["{\colim f_i}", from=1-1, to=1-2]
	\arrow[from=1-1, to=2-1]
	\arrow[from=1-2, to=2-2]
	\arrow[from=2-1, to=2-2]
	\arrow["\lrcorner"{anchor=center, pos=0.125, rotate=180}, draw=none, from=2-2, to=1-1]
\end{diagram}
    Observe that since $\I_1$ is non-empty, it follows that $\LinkI (\str) = \emptyset$. Hence, the left hand vertical in the last pushout square is an isomorphism, showing that the right hand vertical is also an isomorphism. This shows that $\LinkI (- *_{\pos} -)$ preserves colimits in the right argument. The case of the left argument follows by symmetry.
    \end{proof}
\begin{example}\label{ex:Link_of_horn}
    We may use the stratified join to compute the links of stratified horns. Let $\I$ be a regular flag, $\J = [p_0 \leq \dots \leq p_l]$ be some arbitrary flag of $\pos$ and $k \in [l]$. Furthermore, denote by $\overline{\J}$ the unique regular flag from which $\J$ degenerates. Then, the horn inclusion $\Lambda^\J_k \hookrightarrow \Delta^\J$ has the following image under $\LinkI$:
    \begin{enumerate}
        \item If $\I$ is not a subflag of $\overline{\J}$, it is immediate from the definition of $\LinkI$ that \[
        \LinkI \Lambda^\J_k = \emptyset = \LinkI \Delta^\J.\]
        \item If $\I = \overline{\J} \setminus \{p_k\}$ and $\J_{p_k}$ has length $0$, then $\Lambda^\J_k  \cong \partial \Delta^{\J_{\I}} *_P \Delta^{[p_k]}$ and it follows that the image is given by the inclusion \[ \LinkI \Lambda^\J_k = \LinkI (\Lambda^\J_k)_{\I}= \LinkI \partial \Delta^{\J_{\I}} \subset \LinkI \Delta^{\J_{\I}} =\LinkI \Delta^\J. \]
        $\LinkI (\Lambda^\J_k)_{\I}$ is precisely given by the boundary of the polygon $\LinkI \Delta^\J = \prod_{p_i \in \I} \Delta^{\J_{p_i}}$.
        \item If $\I \subsetneq \overline{\J}$, and furthermore $\I \neq \overline{\J} \setminus \{p_k\}$ or the length of $\J_{p_k}$ is not $0$, then we obtain the identity
            \[ \LinkI \Lambda^\J_k = \LinkI \Delta^{\J_{\I}} = \LinkI \Delta^\J.\]
        \item If $ \I = \overline{\J}$, then we may represent $\Lambda^\J_k$ as a join as follows. Denote $\I_0 = \I \setminus \{p_k\}$, $\I_1 = \{ p_k \}$, and by $\J_0$ the restriction of $\J$ to $\I_0$. Let $k_0$ be minimal with the property that $p_{k_0} = p_k$. Denote by $F_k$ the $(k-k_0)$-th face of $\Delta^{\J_{p_k}}$.
        Then 
        \[
        \Lambda^\J_k = ( \Delta^{\J_0} *_P F_k ) \cup_{\partial \Delta^{\J_0} *_P  F_k} ( \partial \Delta^{\J_0} *_P \Delta^{\J_{p_k}}).
        \] Therefore, if we apply $\LinkI$ and use the interaction with stratified joins, we obtain
        \begin{align*}
             \LinkI \Lambda^\J_k &= (\LinkI[\I_0] \Delta^{\J_0} \times F_k) \cup_{\LinkI[\I_0] (\partial \Delta^{\J_0}) \times F_k} (\LinkI[\I_0](\partial \Delta^{\J_0}) \times \Delta^{\J_{p_k}}) \\
             & \subset \LinkI[\I_0]{\Delta^{\J_0}} \times \Delta^{\J_{p_k}} = \LinkI \Delta^\J.
        \end{align*}
    \end{enumerate}
\end{example}
As a consequence of our computations in \cref{ex:Link_of_horn}, we obtain the following corollary, characterizing admissible horn inclusions:
 \begin{corollary}\label{cor:char_of_admissible_incl}
     A stratified horn inclusion $j \colon \Lambda^\J_k \hookrightarrow \Delta^\J$ is admissible if and only if 
     $\LinkI j$ is a weak homotopy equivalence for all regular flags $\I$.
 \end{corollary}
 \begin{proof}
    \cref{ex:Link_of_horn} covers all possible examples of combinations of $\I$ and $\J$. Let $\J$ and $k$ be such that $j$ is admissible. Then, in the first and third cases, the induced map $\LinkI j$ is an isomorphism. The second case cannot occur, as it is assumed that $\J_{p_k}$ has a length greater than or equal to $1$, by the definition of admissibility. Therefore, the only remaining case is the fourth. Note that $F_k \hookrightarrow \Delta^{\J_{p_k}}$ is an acyclic cofibration in the Quillen model structure. Hence, in the fourth case it follows from the description of $\LinkI j \colon \LinkI \Lambda^\J_k \hookrightarrow \LinkI \Delta^\J$ in \cref{ex:Link_of_horn} that $\LinkI j$ is given by the box product of a cofibration and an acyclic cofibration, and hence is also an acyclic cofibration of simplicial sets. Conversely, suppose that $\LinkI j$ is an acyclic cofibration for all $\I$. Then, in particular, $\J_{p_k}$ cannot have length $0$, as this would imply that for $\I = \overline \J \setminus \{p_k\}$ the second case of \cref{ex:Link_of_horn} applies. In this case, $\LinkI$ is given by the boundary inclusion of a polygon, which is not a weak homotopy equivalence.
 \end{proof}
  \begin{proposition}\label{prop:geo_link_is_left_quilen}
            The functor
            \[
            \LinkI[] \colon \sSetPN \to \Diag
            \]
            is the left part of a Quillen adjunction between $\sSetPDN$ and $\DiagSInN$.
        \end{proposition}
        \begin{proof}
            That $\LinkI[]$ admits a right adjoint is immediate from its construction via Kan extension on a category of presheaves. 
            Furthermore, one may easily see that, for any regular flag $\I$, $\LinkI$ sends monomorphisms to pointwise monomorphisms, and hence preserves all cofibrations. A generating set of acyclic cofibrations in $\sSetPDN$ is given by the admissible horn inclusions (\cite[Thm. 2.14]{douSimp}). Therefore, we only need to show that, for any regular flag $\I$ and any admissible horn inclusion $\Lambda^\J_k \hookrightarrow \Delta^\J$, the induced simplicial map 
            \[
            \LinkI{\Lambda^\J_k} \hookrightarrow \LinkI{\Delta^\J}
            \]
            is a weak homotopy equivalence. This is the content of \cref{cor:char_of_admissible_incl}.
        \end{proof}
        Let us now compare $\LinkI$ with the simplicial homotopy link.
        \begin{construction}
        A natural transformation 
        \[
        \LinkI[] \to \HolIPS[]
        \]
        is constructed as follows: Given a flag $\J$ of length $m$ of $P$, which contains a regular flag $\I = \standardFlag$, a $k$-simplex $\tau$ of $\LinkI[\I](\Delta^\J)$ is given by an $n+1$-tuple $(\tau_0, \tau_1 , \cdots , \tau_n)$, with $\tau_i\colon \Delta^k \to \Delta^{\J_{p_i}}$. Under the inclusions $\Delta^{\J_{p_i}} \hookrightarrow \Delta^\J$, we may equivalently interpret these data as a $[m]$ valued matrix $(i_{lj})_{l \in [n], j \in [k]}$, with the properties:
        \begin{itemize}
            \item $q_{i_{lj}} = p_j$ for all $l \in [n], j \in [k]$;
            \item $ i_{lj} \leq i_{l(j+1)}$, for all $l \in [n], j \in [k-1]$.
        \end{itemize}
        As a consequence of the first property, any such matrix also fulfills
        \begin{itemize}
            \item $i_{lj} < i_{{l+1}j}$ for all $l \in [n-1]$, $j \in [k]$.
        \end{itemize}
        Together, the second and the third property imply that
        \begin{itemize}
            \item $i_{lj} \leq i_{l'j'}$, for $l\leq l' \in [n]$ and $j\leq j' \in [k]$.
        \end{itemize}
        Equivalently, such a matrix is precisely the data of a stratum-preserving simplicial map
         \[
        \hat \tau \colon \Delta^\I \times \Delta^{k} \to \Delta^\J
        \]
        given by uniquely extending the map of vertices
        \[
        (l,j) \mapsto (i_{lj}). 
        \]
        One may easily check that this construction is compatible with face and degeneracy maps. Thus, we obtain an induced isomorphism of simplicial sets
        \begin{align*}
            \LinkI(\Delta^\J) &\to \HolIPS(\Delta^\J) \\
            \tau &\mapsto \hat \tau
        \end{align*}
        natural in $\J$ and $\I$ (when $\I$ is not a subflag of $\J$, both simplicial sets are empty by definition). 
        Therefore, again by left Kan extension, we obtain a natural transformation
        \[
        \LinkI[] \to \HolIPS[].
        \]
    \end{construction}
   \begin{proposition}\label{theo:geometric_link}
       The natural transformation $\tau \colon \LinkI[] \to \HolIPS[]$ is given by weak equivalences in $\DiagSInN$.
   \end{proposition}
    We are going to give a purely abstract proof here. Before we do so, let us, however, give a geometrical intuition for why the statement holds. 
    \begin{example}
        Suppose $\pos = \{p < q\}$ is a poset with two strata.
        For a stratified simplex $\Delta^\J$, the image of $\Delta^\J$ under $\LinkI[]$ is the diagram 
        \[
        D= \{ \Delta^{\J}_{p} \leftarrow \Delta^\J_{p} \times \Delta^\J_{q} \to \Delta^\J_{q} \}.
        \]
        If we apply $\int^{\I} \Delta^\I \otimes -_{\I}$ to this diagram, we obtain the quotient of the stratified simplicial set
        \[
        \Delta^\J_{p} \times \Delta^\J_{q} \times \Delta^{[p < q]}
        \]
        obtained by collapsing $\Delta^\J_{p} \times \Delta^\J_{q}$ to $\Delta^\J_{p}$ and $\Delta^\J_{q}$, respectively, at the ends of the interval $\Delta^{[p < q]}$. Note that this construction is just a stratified version of Joyal's alternative join (see, for example, \cite[4.2.1]{HigherCatCisinki}). We obtain a natural comparison map
        \[
        \int^{\I} \Delta^\I \otimes D_{\I} \to \Delta^\J_{p} *_P \Delta^\J_{q} = \Delta^\J.
        \]
        This comparison is natural in $\J$, and we thus obtain a natural transformation
        \[
        \int^{\I} \Delta^\I \otimes \LinkI(-) \to 1_{\sSetPN}.
        \]
        This map is not an isomorphism. However, it is stratified homotopic to a stratified homeomorphism after passing to the topological stratified world. In this sense, $\LinkI[]$ can be thought of as an actual (left) inverse to $\int^{\I} \Delta^\I \otimes -_{\I}$ up to passing from combinatorics to topology. We may just think of this as the statement that a piecewise linear space may be decomposed into a double mapping cylinder along the boundary of some regular neighborhood.
    \end{example}
    \begin{proof}[Proof of \cref{theo:geometric_link}]
        As a consequence of \cref{cor:equ_inj_mod_struct}, $\HolIPS$ preserves homotopy colimits. Since $\LinkI$ is the left part of a Quillen adjunction (with source a cofibrant model category), the same holds for $\LinkI$. 
        As every stratified simplicial set is the homotopy colimit of its stratified simplices (\cref{rem:hocolim_of_simp}), it hence suffices to show that $\tau$ is a weak equivalence on the latter. However, on stratified simplices, $\tau$ is even an isomorphism of simplicial sets.
    \end{proof}
    The fact that, up to weak equivalence, this makes $\HolIPS$ both the left part and the right part of a Quillen equivalence turns out to be quite useful in practice. Let us illustrate this by providing some model structures for d\'ecollages, as defined in \cite{haine2018homotopy}. In particular, this gives an example of how results on abstract stratified homotopy types can be deduced from a deeper understanding of the Douteau-Henriques model structure.
    \begin{recollection}\label{rec:decollage}
        A diagram $D \in \Diag$ is called a \define{d\'ecollage} over $\Pos$, if for every regular flag $\I = \standardFlag$ in $\sd(\pos)$ the induced simplicial map from $D_{\I}$ into the homotopy limit of 
        \[
        D_{p_0} \leftarrow D_{[p_0, p_1]} \rightarrow \cdots \leftarrow D_{[p_{n-1}, p_n]} \rightarrow D_{p_n}
        \]
        is a weak homotopy equivalence. 
        In \cite[Thm. 1.1.7]{haine2018homotopy} the author shows that the homotopy link construction induces an equivalence of $\infty$-categories between abstract stratified homotopy types and d\'ecollages (using a homotopy coherent model of d\'ecollages).
    \end{recollection}
    Let us construct a model structure presenting the $\infty$-category of d\'ecollages. We will need the following observation.
    \begin{observation}\label{obs:mapping_space_from_subobject}
         Observe that, for a subcomplex $K \subset \nerve(\pos)$, the associated simplicial homotopy link diagram $\HolIPS[](K) \in \Diag$ is given by $\emptyset$, at $\I$ with $\Delta^\I \not \subset K$ and by the terminal simplicial set $\Delta^0$ otherwise. Consequently, for any simplicial set $S$, a morphism $\HolIPS[](K) \otimes \Delta^n \to D$ specifies the same data as a morphism from the constant simplicial presheaf on $\sd (K)^{\op
         } \subset \sd(\pos)^{\op}$ with value $\Delta^n$ into $D|_{\sd (K)^{\op}}$.
         It follows that, for $D \in \Diag$, there is a canonical isomorphism
         \[
         \SDiag(\HolIPS[](K),D) \cong \varprojlim_{\I \in \sd (K)^{\op}} D_{\I},
         \]
         where $\sd (K)$ denotes the subcategory of $\sd(\pos)$ given by the simplices of $K$. 
    \end{observation} 
    \begin{notation}
    Given $K \subset \nerve(\pos)$, and $S \in \sSetN$, we denote 
    \[
    K \DiagTen S  := \HolIPS[](K) \otimes S \in \Diag. 
    \]
    This construction defines a functor from the product of the category of subobjects of $\nerve (P)$ with the category $\sSetN$ into $\Diag$. 
\end{notation}
\begin{observation}\label{obs:Diag_ten_gives_limits}
    By \cref{obs:mapping_space_from_subobject} and the simplicial adjunction $\HolIPS[](K) \otimes  - \dashv \SDiag ( \HolIPS[](K), -)$, it follows that morphisms
    \[
    K \DiagTen S \to D
    \]
    are in natural bijection with arrows
    \[
    S \to \varprojlim_{\I \in \sd (K)^{\op}} D_{\I}.
    \]
    In the special case where $K= \Delta^\I$, the category $\sd (K)$ has the terminal object $\Delta^\I$, and we obtain a canonical isomorphism
    \[
    \varprojlim_{\I' \in \sd (K)^{\op}} D_{\I'} = D_{\I}.
    \]
\end{observation}
    \begin{notation}
        Given a flag $\J = [p_0 \leq \dots \leq p_n]\in \Delta_{\pos}$, we denote by $\textnormal{Sp}(\J) \subset \Delta^\J$ the stratified subsimplicial set whose underlying simplicial set is the spine of $\Delta^n$, i.e. the union of all $1$-simplices of the form $\Delta^{\{k,k+1\}}$, for $0 \leq k \leq n-1$.
    \end{notation}
    Observe that the diagrams that one takes a homotopy limit over in \cref{rec:decollage} are precisely the restriction of $D$ to $\sd(\textnormal{Sp}(\J))^{\op}$.
    \begin{construction}
        As $\sSet$ (with the Kan-Quillen model structure) is a left proper, combinatorial, simplicial model category, so is $\DiagSIn$ (\cite[Rem. 2.8.4, A.3.3.2]{HigherTopos}). By \cite[Thm. 4.7]{BarwickLeftRight}, we can therefore localize $\DiagSIn$ with respect to either of the following sets of morphisms:
        \begin{align*}
             \{ \Lambda^\I_k \DiagTen \Delta^0 \hookrightarrow \Delta^\I \DiagTen \Delta^0 &\mid \I \in \sd(\pos), \Lambda^{\I}_k \hookrightarrow \Delta^\I \textnormal{ is inner}\}; \\
        \{ 
        \textnormal{Sp}(\I) \DiagTen \Delta^0 \hookrightarrow \Delta^\I \DiagTen \Delta^0 &\mid \I \in \sd(\pos) 
        \}.
        \end{align*}
        It turns out that these two localizers result in the same left Bousfield localization (see the proof below).
         An injectively fibrant diagram $D$ is then local with respect to these inclusions, if and only if the induced maps
         \[
         D_{\I} \cong \SDiag( \Delta^\I \DiagTen \Delta^0, D) \to \SDiag( \textnormal{Sp}(\I) \DiagTen \Delta^0, D) \cong \varprojlim_{\I' \in \sd (\textnormal{Sp}(\I))^{\op}} D_{\I'},
         \]
         for $\I \in \sd(\pos)$, are weak equivalences.
        The resulting simplicial model category is called the model category of d\'ecollages and denoted $\DiagSDec$. 
    \end{construction}
    Let us show that these two localizers do indeed produce the same localizations:
    \begin{proof}
    We denote the first localizer by $L_0$ and the second by $L_1$.
    It suffices to see that each of the two localizers is contained, respectively, in the set of acyclic cofibrations generated by the other. To this end, observe that within the class of cofibrations, acyclic cofibrations in a model-category are closed under the operations
    \begin{enumerate}
         \item pushouts along monomorphisms\footnote{Of course, they are also closed under more general pushouts, but this will suffice here.};
         \item right cancellation;
         \item composition.
     \end{enumerate}
    Hence, it suffices to see that each element of $L_0$ is generated under these operations by the elements of $L_1$, and vice versa. Next, observe that the functor $K \mapsto K \DiagTen \Delta^0$ (from the category of subobjects of $\nerve(\pos)$) maps such squares that define pushouts in $\sSetPN$ into pushouts.
    Hence, it suffices to see that the class of inner horn inclusions and spine inclusions of subobjects of $\nerve(\pos)$ generate the same class under the three operations 
    \begin{enumerate}
         \item pushouts in $\sSetPN$ along arrows in the category of subobjects of $\nerve(\pos)$;
         \item right cancellation;
         \item composition.
     \end{enumerate}
    Indeed, any spine inclusion $\textnormal{Sp}(\I) \hookrightarrow \Delta^\I$ can be written as a composition of pushouts of inner horn inclusions, along inclusions (see, for example, the proof of \cite[Prop. 1.3.22]{Land}).
     For the converse inclusion, consider the proof of \cite[Lem. 3.5]{JoyalQCvsSS}.
    \end{proof}
    Let us verify that the bifibrant objects of $\DiagSDec$ are indeed precisely such injectively fibrant diagrams that fulfill the d\'ecollage condition.
    \begin{proposition}\label{prop:presents_decollages}
      A bifibrant object $D \in \DiagSIn$ is a d\'ecollage if and only if it is a bifibrant object in $\DiagSDec$. 
    \end{proposition}
    \begin{proof}
        Observe that all objects in $\DiagSIn$ are cofibrant, and thus that bifibrancy is equivalent to fibrancy.
        Under both conditions $D$ is a fibrant object in $\DiagSIn$. By definition, $D$ is fibrant in $\DiagSDec$ if and only if 
       \[
        D_{\I} \to \varprojlim_{\I' \in \sd (\textnormal{Sp}(\I))^{\op}} D_{\I'}
       \]
       is a weak equivalence, for each $\I \in \sd(\pos)$. To show that this is equivalent to being a d\'ecollage, it suffices to show that the right-hand expression computes the homotopy limit of the restriction of $D$ to $\sd(\textnormal{Sp}(\I))^{\op}$.
       Observe that 
       \[
       E \mapsto \varprojlim_{\I' \in \sd (\textnormal{Sp}(\I))^{\op}} E_{\I} = \SDiag( \textnormal{Sp}(\I) \DiagTen \Delta^0, E)\]
       defines a right Quillen functor (since $\DiagSIn$ is a cofibrant simplicial model category). Let us denote this functor by $F$. Equivalently, we may write $F$ as the composition of the right Quillen functor \[
       \varprojlim \colon \FunC ( \sd (\textnormal{Sp})(\I)^{\op} , \sSetN) \to \sSetN\] with the restriction functor along 
       \[
       j \colon \sd (\textnormal{Sp}(\I))^{\op} \to \sd (\pos)^{\op},\] denoted $j^*$. That is, we have $F = \varprojlim \circ j^*$. Observe that $j^*$ is also a right Quillen functor. 
       To see this, we may treat $\sd(\pos)^{\op}$ as a Reedy category, with all morphisms being degree decreasing and apply \cite[Thm 2.7]{barwickReedy}, from which the claim follows. In the following, given a right Quillen functor $G$, we denote by $RG$ its right derived functor.
       As $D$ was assumed to be fibrant, it follows that $\varprojlim_{\I' \in \sd (\textnormal{Sp}(\I))^{\op}} D_{\I'}$ computes the right derived functor of $F$. Hence, we have 
       \[
       \varprojlim_{\I' \in \sd (\textnormal{Sp}(\I))^{\op}} D_{\I'} = R( \varprojlim \circ j^*)(D) = (R \varprojlim) \circ (Rj^*) (D) =( R \varprojlim) j^*D \simeq \ho \varprojlim (j^*D) ,
       \]
       and we have shown that $\varprojlim_{\I' \in \sd (\textnormal{Sp}(\I))^{\op}} D_{\I'}$ computes precisely the homotopy limit in the defining property of a d\'ecollage. 
    \end{proof}
    \begin{theorem}\label{prop:equ_decol_haine}
        The adjunction
        \[
         \intDiag \colon  \SDiag  \rightleftharpoons  \sSetP \colon \HolIPS[]
        \]
        defines a simplicial Quillen equivalence between $\sSetPC$ and $\DiagSDec$, creating weak equivalences in both directions. 
    \end{theorem}
    \begin{proof}
        We are first going to show that $\intDiag$ sends the localizer defining $\DiagSDec$ to weak equivalences in $\sSetPC$. It then follows by the universal property of Bousfield localization that the adjunction in the statement of the theorem is a Quillen adjunction. 
        We then show that $\HolIPS[]$ also preserves all weak equivalences. 
        Since the ordinary unit of adjunction is given by weak equivalences, it follows from $\HolIPS$ preserving weak equivalences that the derived unit is also a weak equivalence. Consequently, the induced Quillen adjunction is a Quillen equivalence, with (ordinary) unit and counit given by weak equivalences (\cref{rem:equ_char_of_we}). To see the statement about $\intDiag$, note that for $K \subset N(P)$ a subcomplex, we have
        \[
        \intDiag [(K \DiagTen \Delta^0)] = K.
        \]
        Hence, the elements of the localizer defining $\DiagSDec$ are mapped to the stratified inner horn inclusions
        \[
        \Lambda^\I_k \hookrightarrow \Delta^\I,
        \]
        which are acyclic cofibrations by definition of the model structure on $\sSetPC$. To show the statement about $\HolIPS$, observe that by \cref{theo:geometric_link} we may equivalently show that $\LinkI[]$ preserves weak equivalences. As every object in $\sSetPC$ is cofibrant, this follows if we can show that $\LinkI[]$ defines a left Quillen functor with respect to the localizations. 
        Again, by the universal property of the left Bousfield localization, it suffices to show that, for $\Lambda^\J_k \to \Delta^\I$ an inner horn inclusion ($\J = [q_0 \leq \cdots \leq q_n]$) that is not also admissible, the induced morphism 
        \[
        \LinkI[](\Lambda^\J_k) \hookrightarrow \LinkI[](\Delta^\J)
        \]
        is a weak equivalence. Let $\overline{\J}$ be the unique non-degenerate flag which $\J$ degenerates from. 
        Denote $\mathcal J_0 := \J \setminus \{p_k\}$ and $\overline{\mathcal \J}_0 = \overline{\J} \setminus \{p_k\} $. Since $\Lambda^\J_k \hookrightarrow \Delta^\J$ is not admissible, we have that $\J_{p_k}$ has length $0$. If we apply \cref{ex:Link_of_horn}, we obtain the following computations of $\LinkI[](\Lambda^\J_k) \hookrightarrow \LinkI[](\Delta^\J)$ at $\I \in \sd(\pos)$:
        \begin{LinkComp}
            \item  \label{proof:LinkCompE} If $\I $ is not a subflag of $\overline{\J}$: \[ \emptyset \hookrightarrow \emptyset ;\] 
            \item \label{proof:LinkComp2} If $\I \subset \overline{\J}$, $\I \neq \overline{\J}$ and $\I \neq \overline{\J}_0$:  \[ \LinkI \Delta^\J \to \LinkI \Delta^\J ;\]
            \item \label{proof:LinkComp3} If $\I = \overline{\J}, \overline{\J}_0$: \[ \LinkI[\overline \J_0] (\partial \Delta^{\J_0}) \hookrightarrow \LinkI[\overline \J_0] (\Delta^{\J_0}).\]
        \end{LinkComp}
        Let us denote the inclusion of \cref{proof:LinkComp3} by $S \hookrightarrow D$.
        Consider the canonical morphisms (adjoint to the identities on $S$ and $D$)
        \begin{align*}
                  \Delta^{\overline{\J}} \DiagTen S \to  \LinkI[] (\Lambda^\J_k) ; \\
                   \Delta^{\overline{\J}} \DiagTen D \to  \LinkI[] (\Delta^\J) \spaceperiod
        \end{align*}
        These morphisms induce a commutative diagram
        \begin{diagram}\label{diag:pushout_with_links}
            \Delta^{\undeg{\J}} \DiagTen S \cup_{\Lambda^{\undeg{\J}}_l \DiagTen S} \Lambda^{\undeg{\J}}_l \DiagTen D  \arrow[d]\arrow[r] & \LinkI[](\Lambda^\J_k) \arrow[d]\\
            \Delta^{\undeg{\J}} \DiagTen D  \arrow[r ]&  \LinkI[](\Delta^\J),
        \end{diagram}
        where $l$ is uniquely determined by $\overline{\J} = [q_0 < \cdots < q_m]$ fulfilling, $q_l = p_k$. We claim that this diagram is pushout. Proving this finishes the proof, since the left vertical is given by a box product of a localizer defining $\DiagSDec$ (namely $\Lambda^{\overline{\J}}_l \DiagTen \Delta^0 \to \Delta^{\overline{\J}} \DiagTen \Delta^0$) with a cofibration of simplicial sets (namely $S \hookrightarrow D$). Let us verify the cocartesianity of this diagram at each $\I \in \sd(\pos)$. If $\I $ is not a subflag of $\J$, then \cref{diag:pushout_with_links} is empty by \cref{proof:LinkCompE}. For $\I \subset \overline{\J}$, $\I \neq \overline{\J}, \overline{\J}_0 $, by \cref{proof:LinkComp2}, both verticals are isomorphisms, which makes the diagram cocartesian. Finally, by \cref{proof:LinkComp3}, if $\I = \overline{\J}, \overline{\J_0}$, then both horizontals are isomorphisms.
        \end{proof}

 \section{Combinatorial models over varying posets}\label{sec:global_mod_struct}
 In this section, we define global analogues of the Douteau-Henriques and the Joyal-Kan model structures described in the previous section. 
\subsection{From local to global model structures}\label{subsec:from_local_to_global}
Now, let us piece the model structures on $\sSetP$, where $\pos$ varies over all posets, together to model structures on $\sStrat$.
To do so, we make use of the following general principle, which is the special case of the characterization of bifibrations over a model category in \cite{CagneMellies}, where the base category carries the trivial model structure (we use the notation of \cite{CagneMellies}). This approach was first used in \cite{douteauEnTop}.
\begin{lemma}[{\cite[Thm. 4.4]{CagneMellies}}]\label{lem:gluing_model_struct}
    Suppose that we are given a Grothendieck bifibration $P \colon   \textnormal {\textbf{M}} \to  \textnormal {\textbf{B}}$. Suppose further that , for every $A \in \textnormal {\textbf{B}}$, 
    the fiber $\textnormal {\textbf{M}}_A$ is equipped with the structure of a model category and that, for every morphism $u: A \to B$ in $\textnormal {\textbf{B}}$, the induced functor $u_! \colon \textnormal {\textbf{M}}_A \to \textnormal {\textbf{M}}_B$ is a left Quillen functor. Then $\textnormal {\textbf{M}}$ carries the structure of a model category with the following defining classes. Let $f: X \to Y$ be a morphism in $\textnormal {\textbf{M}}$: 
    \begin{enumerate}
        \item $f$ is a weak equivalence, if and only if $P(f)$ is an isomorphism and $\baseBack$ is a weak equivalence in $\textnormal {\textbf{M}}_{P(Y)}$ (or equivalently $\baseForw$ is a weak equivalence in $\textnormal {\textbf{M}}_{P(X)}$).
        \item $f$ is a cofibration, if and only if $\baseForw$ is a cofibration in $\textnormal {\textbf{M}}_{P(Y)}$.
        \item $f$ is a fibration, if and only if $\baseBack$ is a fibration in $\textnormal {\textbf{M}}_{P(X)}$.
    \end{enumerate}
    Furthermore, assume that $\textnormal {\textbf{M}}$ is a simplicial category and $P$ a simplicial functor (with respect to the discrete structure on $\textnormal {\textbf{B}}$) such that $u_! \dashv u^*$ is a simplicial adjunction, for all $u \in \textnormal {\textbf{B}}$. If for each $A \in \mathcal B$, the category $\textnormal {\textbf{M}}_A$ is a simplicial model category with respect to the simplicial structure inherited from $\textnormal {\textbf{M}}$, then so is $\textnormal {\textbf{M}}$.
\end{lemma}
To apply \cref{lem:gluing_model_struct} to glue the fiberwise model structures on $\sSetP$, we need the following lemma.
\begin{proposition}\label{prop:basechange}
    For any morphism of posets $u \colon \pos \to \pos'$, the induced adjunction \begin{align*}
        u_! \colon \sSetP \rightleftharpoons \sSetP[\pos'] \colon u^*
    \end{align*}
    - given by postcomposition and pulling back along $u$ - is a simplicial Quillen adjunction, with respect to the Douteau-Henriques and the Joyal-Kan model structures (taken the same on both sides, respectively). Furthermore, again in both scenarios, $u_!$ reflects fibrations and creates acyclic fibrations.
\end{proposition}
\begin{proof}
    Simpliciality is immediate by definition.
    Clearly, $u_!$ preserves all cofibrations. Furthermore, $u_!$ preserves admissible horn inclusions, which shows the case of the Douteau-Henriques model structure, as the latter generate the acyclic cofibrations. For the case of the Joyal-Kan model structure, by \cref{prop:loc}, it suffices to show that $u_!$ sends stratified inner horn inclusions to acyclic cofibrations. Clearly, the image of every stratified inner horn inclusion under $u_!$ remains an inner horn inclusion. This shows that $u_!$ is left Quillen. Now to see that $u_!$ reflects (acyclic) fibrations, note that for any lifting diagram 
    \begin{diagram}\label{diag:lift_basechange}
        {\str[A]} \arrow[r] \arrow[d] & \str[X] \arrow[d] \\
        \str[B] \arrow[r] & \str[Y]
    \end{diagram}
    any dashed solution to 
    \begin{diagram}
        {u_!\str[A]} \arrow[r] \arrow[d] & u_!\str[X] \arrow[d] \\
        u_!\str[B] \arrow[r] \arrow[ru, dashed]& u_!\str[Y]
    \end{diagram}
    already provides a solution to \cref{diag:lift_basechange}. Indeed, commutativity at the level of simplicial sets already implies that $\str[B] \to \str[X]$ is stratum-preserving, as all these diagrams can be considered in the slice category over $\str[Y]$, which is independent from the stratifications. Hence, the reflection property follows from $u_!$ being left Quillen. That $u_!$ preserves acyclic fibrations follows similarly, if we note that every lifting diagram 
     \begin{diagram}
        {\str[A]} \arrow[r] \arrow[d, hook] & {u_!\str[X]} \arrow[d] \\
        \str[B] \arrow[r] \arrow[ru, dashed]& u_!\str[Y]
    \end{diagram}
    lies in the image of $u_!$ and that $u_!$ creates cofibrations.
\end{proof}
\begin{definition}\label{def:non_ref_mod_struct}
    We denote by $\sStratD$ and $\sStratC$ the simplicial model categories with underlying category $\sStrat$, defined by applying \cref{lem:gluing_model_struct} to the forgetful functor \[\sStrat \to \Pos,\] with the fiberwise model structures given by $\sSetPD$ and $\sSetPC$, for $P \in \Pos$, respectively.
    The model structure on $\sStratD$ is called the \define{Douteau-Henriques model structure} on $\sStrat$. The model structure on $\sStratC$, is called the \define{Joyal-Kan model structure} on $\sStrat$.
    Weak equivalences in these model categories are called \define{poset-preserving diagrammatic equivalences} and \define{poset-preserving Joyal-Kan equivalences}, respectively.
\end{definition}
Let us begin our investigation of these model structures with the following observation:
\begin{lemma}\label{lem:basechange_fib}
    Let $\str \in \sStratN$ and let $f \colon Q \to \pstr, g \colon \pstr \to Q' \in \Pos$. Then the induced natural map $f^*\str \to \str$ is a fibration and the natural map $\str \to  g_!\str$ is a cofibration, in $\sStratCN$ and $\sStratDN$.
\end{lemma}
\begin{proof}
    This is immediate from the simple observation that $\baseBack[f]$ and $\baseForw[g]$ are both given by isomorphisms (the identity even).
\end{proof}
\begin{proposition}\label{prop:strat_we_stable_colim}
    Weak equivalences in $\sStratC$ and $\sStratD$ are stable under filtered colimits. 
\end{proposition}
\begin{proof}
    Note that as every weak equivalence is given on posets by an isomorphism, and filtered diagrams lack monodromy, it follows that the colimit of all posets involved is canonically isomorphic to any of the posets in the filtered diagram, and we may easily reduce the statement to such diagrams of weak equivalences, which are given by the identity on the poset level. Now, the result follows from \cref{prop:we_stable_colim}.
\end{proof}
\begin{remark}\label{rem:ac_fib_sStrat}
Note that the acyclic fibrations in $\sStratD$ and $\sStratC$ are precisely the stratified maps that induce an isomorphism on posets and an acyclic fibration (in the Joyal or Kan model structure) on simplicial sets. Indeed, this follows by applying \cref{prop:basechange} to $u \colon \pos \to [0]$.
\end{remark}
We may then state the following global version of \cite[Cor. 2.5.11]{haine2018homotopy}. 
\begin{recollection}
     Recall by \cite{HigherTopos} that the quasi-category of all (small) quasi-categories $\iCat$ is given by the homotopy coherent nerve of the simplicial category $\siCat$, whose objects are small quasi-categories $X$, and whose mapping spaces are given by the Kan complexes $\sSetL(X,Y)^{\simeq}$, given by the maximal Kan complex in $\sSetL(X,Y)$. The infinity category of \define{abstract stratified homotopy types} (see \cite{haine2018homotopy}), denoted $\AbStr$, is the full subcategory of the arrow quasi-category $ \iCat^{\Delta^{1}}$ of conservative functors $F \colon X \to \pos$, where $\pos$ is a poset. (The tilde over $\sSet$ indicates that, in order to avoid set-theoretic issues, $\sSetL$ is modeled on a larger Grothendieck universe than $\sSet$.)
\end{recollection}
\begin{proposition}\label{prop:sStraC_pres_Astrat}
    $\sStratC$ presents the $\infty$-category of abstract stratified homotopy types. 
\end{proposition}
\begin{proof}
    A stratified simplicial set $\str \in \sStratCN$ is fibrant, if and only if it is fibrant as an element of $\sSetPCN$, with $P= \pstr$. It follows that the bifibrant objects of $\sStratCN$ are precisely the abstract stratified homotopy types. 
    Next, consider the simplicial functor category $\siCat^{[1]}$, i.e. the simplicial category of arrows in $\siCat$.
    Given two such fibrant objects $\str$ and $\str[Y]$, $\siCat^{[1]}( \str ,\str[Y])[n]$ is given by the set of such morphisms  \[ F \colon( \ustr \times \Delta^n \to \pstr \times \Delta^n) \to (\ustr[Y] \to \pstr[Y]) \]
    which fulfill
    \[
    F_0(\{x\} \times \Delta^n) \subset \ustr[Y]^{\simeq} \textnormal{ and } F_1(\{p\} \times \Delta^n) \subset \pstr[Y]^{\simeq} \textnormal{ for } x\in \ustr,p \in P.
    \] Note that since $\sstr[Y]$ is a conservative functor, the condition that $F(\{x\} \times \Delta^n) \subset \ustr[Y]^{\simeq}$, for $x \in \ustr$, is redundant.
    Furthermore, $P ^{\simeq}$ is discrete (every isomorphism in a poset is the identity). Therefore, the condition $F_1(\{p\} \times \Delta^n)\subset \pstr[Y]^{\simeq}$ is equivalent to saying that $F_1 \colon \pos \times \Delta^n \to \pstr[Y]$ is of the form $\pos \times \Delta^n \to \pos \xrightarrow{u} \pstr[Y]$. Hence, $\siCat^{[1]}( \str ,\str[Y])[n]$ is equivalently the set of stratified maps 
    \[
    \str \otimes \Delta^n \to \str[Y]
    \]
    which is precisely $\sStrat(\str,\str[Y]) [n]$. To summarize, we have shown that if we denote by  
    $\sStrat^{o}$ the simplicial category of bifibrant objects in $\sStratC$, 
    then $\sStrat^{o}$ is even isomorphic to the full subcategory of $\siCat^{[1]}$ given by conservative functors into a poset. Making use of this, we treat $\sStrat^{o}$ as a full subcategory of $\siCat^{[1]}$.
    Denote by $(\siCat^{[1]})^{o}$ the full subcategory of $\siCat^{[1]}$ given by such functors $f \colon X \to Y$ that are an iso-fibration, i.e. the full simplicial subcategory of bifibrant objects in the injective model structure. $(\siCat^{[1]})^{o}$ is a model for the category of arrows in $\iCat$ in terms of simplicial categories. More precisely, if we denote by $\mathcal{I}\cat[sofib]$ the full subcategory of $\iCat^{\Delta^1}$ of isofibrations, then there is a natural zig-zag of Joyal-equivalences
    \[
    \iCat^{\Delta^1} \xhookleftarrow{\simeq} \mathcal{I}\cat[sofib] \xrightarrow{\simeq}\nerve \big{(}(\siCat^{[1]})^o \big{)} \spaceperiod
    \]
    The left-hand side equivalence follows from fibrant replacement in the Joyal model structure. The right-hand equivalence is induced by the natural transformations $\mathcal{S}( \Delta^n \times \Delta^1) \to \mathcal{S}( \Delta^n) \times [1]$, where $\mathcal{S}$ is the left-adjoint of the homotopy coherent nerve, and is a weak equivalence by \cite[A.3.4.13.]{HigherTopos} applied to the model structure of marked simplicial sets presenting $(\infty,1)$-categories.
    Note that every conservative functor from a quasi-category into a poset is necessarily an isofibration. Indeed, every functor with target the nerve of a $1$-category is an inner fibration, and every isomorphism in a poset is the identity, which clearly admits a lift. It follows that $\AbStr \subset \mathcal{I}\cat[sofib]$ as a full subcategory and that $\sStrat^{o} \subset (\siCat^{[1]})^{o}$. We thus obtain a commutative square
    \begin{diagram}{}
        \mathcal{I}\cat[sofib] \arrow[r, "\simeq"] & \nerve \big ((\siCat^{[1]})^{o} \big ) \\
        \AbStr \arrow[u, hook ,"f.f."] \arrow[r, dashed] & \nerve (\sStrat^{o}) \arrow[u, "f.f.", hook]\spacecomma
    \end{diagram}
    with the lower horizontal induced by the fact that the composition of the left vertical and right horizontal has image in $\nerve (\sStrat^{o})$. This dashed functor even is a bijection on objects, as we have already noted in the beginning of this proof. Furthermore, by commutativity of the diagram, it is fully faithful. 
    Hence, we have constructed an equivalence of quasi-categories $\AbStr \simeq \nerve (\sStrat^{o})$ as claimed.
\end{proof}
Next, we gather some general properties of the model categories $\sStratC$ and $\sStratD$. We begin with the following general lemma.
\begin{lemma}\label{lem:gen_for_global_mod}
    In the situation of \cref{lem:gluing_model_struct}, assume that $P : \textnormal {\textbf {M}} \to \textnormal {\textbf{B}}$ admits a left adjoint $L \colon \textnormal {\textbf{B}} \to \textnormal {\textbf{M}}$. Furthermore, let $S$ be a set of morphisms in $\textnormal {\textbf{B}}$, such that a morphism $u$ in $\textnormal {\textbf{B}}$ is an isomorphism if and only if it has the right lifting property with respect to $S$. Let $I$ be a set of (acyclic) cofibrations in $\textnormal {\textbf{M}}$ such that:
    \begin{enumerate}
        \item Each $i \in I$ is contained in some fiber $\textnormal {\textbf{M}}_A$, for some $A \in \textnormal {\textbf{B}}$.
        \item For each $A \in \textnormal {\textbf{B}}$, the set
            \[
                \{ u_!i \mid i \in I, u \colon B \to A, B \in \textnormal{\textbf{B}} \} \]
                is a set of generating (acyclic) cofibrations for $\textnormal {\textbf{M}}_A$.
    \end{enumerate}
    Then $L(S) \cup I$ ($I$) is a set of generating (acyclic) cofibrations for $\textnormal {\textbf{M}}$.
\end{lemma}
\begin{proof}
    We prove the case of cofibrations. Consider a morphism $f \colon X \to Y$ in $\mathcal{M}$. We need to show that $f$ is an acyclic fibration (i.e. $P(f)$ is an isomorphism and $\baseBack$ is an acyclic fibration in $\mathcal{M}_{P(X)}$), if and only if $f$ has the right lifting property with respect to $L(S) \cup I$. 
    Note that by the adjunction $L \dashv P$, the map $P(f)$ is an isomorphism, if and only if $f$ has the right lifting property with respect to $L(S)$. 
    Hence, in the following we may assume without loss of generality that $P(f) = 1_X$.
    Next, note that since $P$ is a Grothendieck left fibration, any lifting problem
        \begin{diagram}
            X_0 \arrow[d, "i"] \arrow[r, "g"] &  X \arrow[d, "f"] \\
            X_1 \arrow[r]\arrow[ru, dashed]& Y
        \end{diagram}
        with $P(i)$ an identity is equivalent to a unique lifting problem 
        \begin{diagram}
            P(g)_! X_0 \arrow[d, "P(g)_! i"'] \arrow[r] &  X \arrow[d, "f"]\\
            P(g)_!X_1 \arrow[r]\arrow[ru, dashed]& Y.
    \end{diagram}
    Hence, as $\{ u_!i \mid i \in I, u \colon B \to A, B \in \mathcal{A} \} $ is a set of generating cofibrations for $\mathcal{M}_{P(X)}$, follows that $\baseBack$ is an acyclic fibration, if and only if $f$ has the right lifting property with respect to $I$.
\end{proof}
\begin{corollary}\label{cor:cof_gen_sStratD}
    The model category $\sStratD$ is cofibrantly generated. A generating set of cofibrations is given by the set of stratified boundary inclusions $\{ \stratBound  \hookrightarrow \stratSim \mid n \in \mathbb N\}$, together with the two morphisms
    \begin{diagram}
        \emptyset \arrow[r]  \arrow[d] & \emptyset \arrow[d] & & \emptyset  \arrow[r]  \arrow[d] & \emptyset \arrow[d] \\
        \emptyset \arrow[r] & {[0]}                           \spacecomma& &{[0] \sqcup [0]} \arrow[r, hook] & {[1]} \spaceperiod
    \end{diagram}
    A generating set of acyclic cofibrations for $\sStratD$ is given by the set of admissible horn inclusions
    \begin{diagram}
        \Lambda_k^n \arrow[rr] \arrow[rd] & &\Delta^n \arrow[ld] \\
        & {[m]} &,
    \end{diagram}
    for $n,m \in \mathbb N$. 
\end{corollary}
\begin{proof}
    This is a consequence of \cref{lem:gen_for_global_mod}. Note that $i \colon \lstr (\partial \Delta^1 \hookrightarrow \Delta^1) =(\stratBound[1] \hookrightarrow \stratSim)$ while being a cofibration, is not contained in a fiber of $\sStrat \to \Pos$. However, we may replace $i$ by the stratified simplicial map obtained by pushing out along the stratified simplicial map 
    \begin{diagram}
        \partial \Delta^1 \arrow[r, "1"] \arrow[d, "{\sstr[]}_{\lstr ( \partial \Delta^1)}"]& \partial \Delta^1  \arrow[d, hook] \\
       { [0] \sqcup [0]} \arrow[r, hook] & {[1]} \spaceperiod
    \end{diagram}
    Denote by $I$ the class of cofibrations in $\sStrat$ obtained in this manner. 
    Now, let us verify the requirements of \cref{lem:gen_for_global_mod}. First, note that the forgetful functor $\sStrat \to \Pos$ admits a left adjoint given by mapping $\pos$ to $ \emptyset \to \pos$. The two morphisms of posets $\emptyset \to [0]$ and $[0] \sqcup [0] \hookrightarrow [1]$ detect all isomorphisms of posets. Indeed, the former detects surjectivity. The latter detects surjectivity on relations. Any morphism of posets which is surjective on points and relations is necessarily an isomorphism. For any $P \in \Pos$, the stratified boundary inclusions over $\pos$, together with the admissible horn inclusions, form sets of (acyclic) cofibrant generators (see \cref{rec:Dou-Hen-mod}). Clearly, the elements of these sets are respectively of the form $u_!(i)$, for $i \in I$ or $i$ an admissible horn inclusion as in the statement of the corollary, where $u$ is an appropriate map of posets with target $\pos$.
\end{proof}
Next, we show that $\sStratC$ is cofibrantly generated. However, since we lack an explicit set of acyclic generators for $\sSetPC$, some additional work needs to be done to show that there is a set of acyclic generators for $\sStratC$.
We are going to take a slight detour to see this. As a corollary of \cref{prop:loc}, we have:
\begin{proposition}\label{prop:c_is_left_bous_of_d}
    $\sStratC$ is the left Bousfield localization of $\sStratD$ at the class of inner stratified horn inclusions
   \[ \stratHorn \hookrightarrow \stratSim\]
    for $0 < k < n$.
\end{proposition}
We may then conclude:
\begin{proposition}\label{prop:cof_comb}
    The simplicial model categories $\sStratC$ and $\sStratD$ are cofibrant and combinatorial.
\end{proposition}
\begin{proof}
    Cofibrancy is obvious. It is not hard to see that $\sStrat$ is generated by the sources and targets of the generating cofibrations in \cref{cor:cof_gen_sStratD}, under filtered colimits. It follows that $\sStrat$ is finitely locally presentable. Since $\sStratD$ is cofibrantly generated, we may hence conclude that $\sStratD$ is combinatorial. As $\sStratC$ is a left Bousfield localization of $\sStratD$ at a set of morphisms, it follows by \cite[Thm. 4.7]{BarwickLeftRight} that $\sStratC$ is also combinatorial.
\end{proof}
\subsection{Model structures of refined stratified simplicial sets}\label{subsec:refining_strat_sset}
For classical examples of stratified spaces the stratification poset is usually strongly related to the topology of the underlying space. In fact, originally, the poset structure arises from the closure containment relation on a partition of a space into disjoint subsets (\cite{mather1970notes}). For general stratified simplicial sets, $\str$, the only relationship between the underlying object and $\pos$ is that the existence of an edge $x \to y$ implies a relation $\sstr(x) \leq \sstr(y)$. 
This degree of generality is, of course, necessary when we are working over a fixed poset, at least if we want to have access to all stratified simplices over $\Pos$.
If we allow for flexible posets, however, then this amount of generality has some peculiar side effects. In fact, we may take it to the extreme as follows:
\begin{remark}\label{rem:so_many_empty_spaces}
    Denote by $L \colon \Pos \to \sStratN$ the left adjoint to the forgetful functor $\sStratN \to \Pos$, given by $\pos \mapsto ( \emptyset \to \pos)$. Clearly, $L$ is fully faithful. If we equip $\Pos$ with the trivial model structure (in which weak equivalences are precisely the isomorphisms, and all maps are cofibrations and fibrations), then $L$ becomes a left Quillen functor with target $\sStratCN$ ($\sStratDN$). One may then verify that $L$ induces a fully faithful embedding $\Pos \hookrightarrow \ho \sStratCN$ ($\ho \sStratDN$). In other words, the homotopy category $\ho \sStratCN$ contains a complete copy of $\Pos$, consisting of empty stratified simplicial sets.    
\end{remark}
We may aim for a notion of stratified simplicial sets for which the poset structure is minimal in some sense, which at least should imply that maps are uniquely determined on the level of simplicial sets. To do so, let us first consider the following functor:
\begin{construction}\label{con:posetification}
    Consider the fully faithful inclusion
    \[
    \Pos \hookrightarrow \sSetN
    \]
    given by taking the nerve of a poset. It admits a left adjoint, which we denote $P$, explicitly constructed by sending $K$ to the poset generated from $K_0$ by adding the relation
    \[
    x \leq y \iff \exists f \colon x \to y \textnormal{ with $f \in \tau(X)$.}
    \]
    where $\tau(X)$ is the homotopy category of $X$. In particular, this means that whenever there are arrows $f \colon x \to y$, $g \colon y \to x$ in $\tau(X)$, then $x = y$ in $P(X)$.
\end{construction}
\begin{remark}
    One should be careful to note that there is a certain overload of notation here. Namely, there are two ways of associating to a stratified simplicial set $\str$ a poset. We may either associate to it the poset $\pstr$, or the poset $P(\ustr)$. These two posets will generally be different, as is evident from the fact that $P(\ustr)$ does not depend on the stratification of $\str$.
\end{remark}
It follows immediately from the definition in \cref{con:posetification} that the construction factors through taking homotopy categories and one obtains:
\begin{lemma}\label{lem:pos_of_cat_eq}
    Let $f \colon X \to Y$ in $\sSetN$ be a categorical equivalence. Then $P(f)$ is an isomorphism.
\end{lemma}
\begin{remark}
Now, to remove redundancies in the stratification poset, at first glance, one may try to invert the stratified maps 
\begin{diagram}
    \ustr \arrow[d] \arrow[r, "1"]& \ustr \arrow[d, "\sstr"] \\
    P(\ustr) \arrow[r] & \pstr.
\end{diagram}
This does, however, not lead to a meaningful homotopy theory of stratified spaces.
Denote the functor $\mathcal X \mapsto (\ustr \to P( \ustr))$ by $(-)^{\red}$. Consider a stratified simplicial set $\str$, and consider $\ustr$ as a trivially stratified simplicial set. Then, if we invert $\str^{\red} \to \str$ and $\ustr^{\red} \to \ustr$, we obtain weak equivalences
\[
\ustr  \simeq  \ustr ^{\red} = \str ^{\red}  \simeq \str,
\]
in other words: We forget all stratifications and simply recover classical homotopy theory.
 Instead, we need to work with a derived version of the functor $\str \mapsto P(\ustr )$, which remembers which paths are within a stratum and should be considered invertible.
\end{remark}

\begin{proposition}\label{prop:computing_der_poset_pre}
    Let $\str,\str[Y] \in \sStrat$ such that all strata of $\str$ and $\str[Y]$ are Kan complexes.
    Then, for any weak equivalence $f \colon \str \to \str[Y]$ in $\sStratC$ the induced morphism of posets $P( \ustr) \to P(\ustr[Y])$ is an isomorphism. 
\end{proposition}
\begin{proof}
Without loss of generality, we may assume that $f$ is the identity on posets $\pstr[X] = \pstr[Y] = \pos$, i.e., that $f$ is a weak equivalence in $\sSetPC$. $\sSetPC$ is equivalently constructed by localizing the model structure on the overcategory $\sSet_{/P}$, coming from the Joyal model structure on $\sSet$, at inclusions $\Delta^{[p \leq p]} \hookrightarrow \Delta^{[p \leq p \leq p]} \cup_{\Delta^{\{0,2\}}} \Delta^{[p]}$, for $p \in \pos$. Indeed, being local with respect to these inclusions precisely means that every morphism in the fibers is an isomorphism, that is, that $\ustr \to \pos$ induces a conservative functor of infinity categories (after fibrantly replacing $\ustr$). It follows that $\str$ and $ \str[Y]$ are local with respect to these inclusions. Hence, $f$ is a Joyal-Kan equivalence in $\sSetP$, if and only if the underlying simplicial map $\ustr \to \ustr[Y]$ is a categorical equivalence (\cite[Prop. 6.3]{nlab:left_bousfield_localization_of_model_categories}). 
Consequently, $f$ induces an equivalence of homotopy categories $\tau( \ustr ) \to \tau( \ustr[Y])$. It follows by construction of $P \colon \sSetN \to \Pos$ that the induced morphism $P(f)$ is an isomorphism.
\end{proof}
In particular, the right derived functors of $P \circ \forget \colon \sStratC, \sStratD \to \Pos$ agree and may be computed by only replacing strata by Kan complexes.
\begin{notation}
    We denote the right derived functor (with respect to $\sStratC$ or $\sStratD$) of the composition
    \[
    \sStrat \xrightarrow{\forget} \sSet \xrightarrow{P(-)} \Pos
    \]
    by $\rpstr[-]$. For $\str \in \sStrat$, we call $\rpstr$ the \define{refined poset} associated to $\str$.
\end{notation}
As an immediate corollary of \cref{prop:computing_der_poset_pre} we obtain:
\begin{corollary}\label{prop:computing_der_poset}
    Let $\str\in \sStrat$ such that all strata of $\str$ are Kan complexes.
    Then, the canonical map 
    \[ P(\ustr) \to \rpstr[\str],\]
    is an isomorphism.
\end{corollary}
We obtain the following explicit description of $\rpstr$.
\begin{proposition}\label{prop:explicit_rp}
    Let $\str \in \sStratN$. Then the underlying set of $\rpstr$ is the set of path components of non-empty strata of $\str$. Furthermore, for any two such components
    $[x]$ and $[y]$, for $x,y \in \ustr$, there is a relation $x \leq y$, if and only if there is a path of $1$-simplices 
    \[
    x = x_0 \leftrightarrow x_1 \leftrightarrow x_2 \leftrightarrow \cdots \leftrightarrow x_n = y
    \] 
    where only simplices that are contained within a stratum of $\str$ are allowed to point in direction of $x$.
\end{proposition}
The following lemma follows from the explicit description \cref{prop:explicit_rp}.
\begin{lemma}\label{lem:refined_poset_pres_colim}
    The functor $\rpstr[-] \colon \sStratN \to \Pos$ preserves filtered colimits. 
\end{lemma}
\begin{construction}
    For $\str \in \sStratN$ we denote by $\str^{\ared}$, its so-called \define{refinement}, is given by the canonical simplicial map $\ustr \to \rpstr$ that maps a vertex to the path component of its stratum (using the explicit construction of $\rpstr$ as in \cref{prop:explicit_rp}). This construction induces an idempotent functor 
    \[
    (-)^{\ared} \colon \sStrat \to \sStrat
    \]
    together with a natural transformation $\str^{\ared} \to \str$, given by 
    \begin{diagram}
        \ustr \arrow[r, "1_{\ustr}"] \arrow[d] & \ustr \arrow[d] \\
        \rpstr \arrow[r] & \pstr
    \end{diagram}
    where the lower map maps a path component to the stratum it is contained in. 
\end{construction}
\begin{definition}\label{def:refined}
    A stratified simplicial set $\str \in \sStratN$ is called \define{refined} if the natural stratified map $\str^{\ared} \to \str$ is an isomorphism.
\end{definition}
Being refined may be interpreted as being stratified in a way that uses the minimal poset (in the sense of minimal amounts of elements and relations) capable of reflecting the same stratified topology (see \cite{TSHHWa}, for topological characterizations). 
\begin{remark}
    Note that by \cref{prop:explicit_rp}, it follows that a stratified simplicial set $\str \in \sStratN$ is refined if and only if $\sstr \colon X \to \pstr$ does not have empty strata, and whenever there is a relation $\sstr(x) \leq \sstr(y)$, for $x,y \in \ustr ([0])$ there is a sequence
    \[
    x = x_0 \leftrightarrow x_1 \leftrightarrow x_2 \leftrightarrow \cdots \leftrightarrow x_n = y
    \]
    of $1$-simplices in $\ustr$, with $\sstr(x)=p$ and $\sstr(y)=q$, and such that only simplices that are contained in one stratum are allowed to point in the direction of $x$. In particular, all strata are path connected. 
\end{remark}
\begin{remark}\label{rem:0-connected_vs_ref}
    If $\str \in \Strat$ is fibrant in $\sStratC$, that is, given by a quasi-category $\ustr$ together with a conservative functor $\ustr \to \pstr$ then being refined is equivalent to being $0$-connected, in the sense of \cite[Def. 2.3.6]{Exodromy}. See also \cref{rem:0-connected_vs_ref_present}.
\end{remark}
One may now easily verify the following:
\begin{proposition}\label{prop:refinement_is_ra}
    The refinement functor $\str \mapsto \str^{\ared}$ has image in the full simplicial subcategory of refined stratified simplicial sets $\str^{\ared}$. It induces the right adjoint to the inclusion of refined stratified simplicial sets into all stratified simplicial sets. The counit of adjunction is given by the refinement morphisms $\str^\ared \to \str$.
\end{proposition}

By \cref{lem:refined_poset_pres_colim} we have:
\begin{lemma}\label{lem:red_pres_filt_colim}
    The functor $(-)^{\ared} \colon \sStrat \to \sStrat$ preserves filtered colimits. 
\end{lemma}
Let us begin by investigating how $(-)^{\ared}$ interacts with the model structures on $\sStrat$.
\begin{construction}
     We will make use of the stratified $\Exi$ functors of \cite[Def. 3.7]{douteauwaas2021} (these were referred to with a \quotes{\textit{naiv}} in \cite{douteauwaas2021}). 
     Denote by $\sd \colon \sSetN \to \sSetN$ the barycentric subdivision functor and by $\Ex$ its right-adjoint (see \cite{KanExInf}). These constructions are extended to stratified spaces as follows:
     For $\str \in \sStratN$, we denote by $\Ex \str$ the stratified simplicial set obtained by the left vertical in the pullback square
    \begin{diagram}
        \forget ({\Ex \str}) \arrow[r] \arrow[d] & \Ex \ustr \arrow[d]  \\
        \pstr \arrow[r, hook] & \Ex \pstr \spaceperiod
    \end{diagram}
    This construction induces a right adjoint to the stratified subdivision functor $\str \mapsto (\sd \ustr \to \ustr \to \pstr)$ .
    There is a natural inclusion $\str \hookrightarrow \Ex \str$ adjoint to the stratified last vertex map $\sd \str \to \str$.
    We denote by $\Exi \str$ the colimit of the diagram \[ \str \hookrightarrow \Ex \str \hookrightarrow \Ex^2\str \hookrightarrow \cdots \spaceperiod\]
    One may easily verify that $\Exi$ is compatible with taking strata, in the sense that $(\Exi \str)_p = \Exi(\ustr_p)$, for $p \in \pstr$.
    It follows from the classical results of \cite{KanExInf} that $\Exi \str$ has strata given by Kan complexes.
    We have shown in \cite[Prop. 3.9]{douteauwaas2021} that the natural inclusion $\str \hookrightarrow \Exi \str$ is an acyclic cofibration in $\sStratD$. In particular, we can compute \[ \rpstr = P ( \forget({\Exi \str} )), \] for all $\str \in \sStrat$.
\end{construction}
\begin{proposition}\label{prop:ared_pres_things}
    The functors
    \begin{align*}
        (-)^{\ared} \colon \sStratD &\to \sStratD ; \\
          (-)^{\ared} \colon \sStratC &\to \sStratC 
    \end{align*}
    preserve cofibrations, acyclic fibrations, and acyclic cofibrations. In particular, they preserve weak equivalences.
    Furthermore, a cofibration $j$ that induces an isomorphism on posets is acyclic if and only if $j^{\ared}$ is an acyclic cofibration.
\end{proposition}
\begin{proof}
    To see the statement concerning acyclic fibrations, note that both model categories we are concerned with have the same acyclic fibrations, and by \cref{rem:ac_fib_sStrat} these are precisely given by such morphisms which induce isomorphisms on posets, and acyclic fibrations in the Joyal model structure on the underlying simplicial sets. Hence, we only need to show that for an acyclic fibration $\str \to \str[Y]$, without loss of generality over the same poset $\pos$, the induced map $\rpstr \to \rpstr[Y]$ is an isomorphism. Now, just as in the classical scenario \cite{KanExInf}, one may show that the functor $\Exi \colon \sStratN \to \sStratN$ preserves acyclic fibrations. Therefore, we may assume without loss of generality that $\str$ and $\str[Y]$ have strata given by Kan complexes, and hence that $\rpstr = P (\ustr)$ and $\rpstr[Y] = P (\ustr[Y])$. As any acyclic fibration is a categorical equivalence in $\sSetN$, it follows that $f$ induces an isomorphism $P(\ustr) = \rpstr \to \rpstr[Y] = P (\ustr[Y])$, as was to be shown. 
    It remains to show the statement on acyclic cofibrations. Clearly, $(-)^{\ared}$ creates cofibrations, as these are defined only in terms of the underlying simplicial sets. Thus, it suffices to show that a cofibration $j \colon \str[A] \hookrightarrow \str[B]$ that is an isomorphism on posets is a weak equivalence, if and only if $j^{\ared}$ is a weak equivalence. Without loss of generality, we may assume that $j$ is the identity on posets. Furthermore, since $(-)^{\ared}$ is given by $\rpstr[-]$ on the posets level, which is a derived functor, it follows that $j^{\ared}$ also is given by an isomorphism on posets. Hence, we also assume that $j^{\ared}$ is given by the identity on the latter. Denote by $Q = \rpstr[B]$ and by $\mathcal Q$ the stratified simplicial set given by $Q \to \pos$. We may thus consider $j$ as an object of the slice category $(\sSetP)_{/ \str[Q]}$. We may then instead show the following stronger claim: The isomorphism of simplicial categories
    \begin{align*}
        (\sSetP)_{/\mathcal Q} \to \sSetP[Q]
    \end{align*}
    is an isomorphism of model categories, where on the left-hand side we use the slice model structure (with respect to $\sSetPD$ or $\sSetPC$).\\ 
    Note that as this is an isomorphism of simplicial categories, and the cofibrations in all categories involved are given by monomorphisms, it suffices to show that the isomorphism identifies the classes of fibrant objects.
    On the left-hand side, these are given by fibrations $\str \to \str[Q]$ (respectively in $\sSetPD$ and $\sSetPC$). By \cref{prop:explicit_rp}, the map $f \colon Q \to \pos$ has fibers which contain no relations, but the identity. In other words, the functor $f\colon Q \to \pos$ is conservative. It follows from this (using \cite[Prop. 2.4.1.5]{HigherTopos}), that $f$ has the right lifting property with respect to all inner and admissible horn inclusions. Hence, $\str[Q]$ is a fibrant object of $\sSetPD$ and of $\sSetPC$. Consequently, we only need to show that $\str$ being fibrant implies $\str \to \str[Q]$ being a fibration, in both scenarios. Since $\str[Q]$ is fibrant in both scenarios, fibrancy of $\str \to \str[Q]$ can be checked by having the right lifting property with respect to admissible, and inner and admissible horn inclusions, respectively.
    Now, consider a lifting diagram 
    \begin{diagram}
        \Lambda^\J_k \arrow[dd,hook] \arrow[r]& \str \arrow[d]\\
                                    & \str[Q] \arrow[d]\\
        \Delta^\J  \arrow[ru]  \arrow[r]             \arrow[ruu, dashed]           & P
    \end{diagram}       
    where the solid part of the diagram commutes, and the dashed map makes the upper and lower triangle of the outer rectangle diagram commute. Furthermore, assume that the left vertical horn inclusion is either admissible or inner.
    To finish the proof, it suffices to prove that this also implies that the middle triangle commutes.
    Since $\ustr[Q]$ is a simplicial complex, it suffices to verify commutativity on vertices. 
    If a vertex $x \in \Delta^\J$ lies in $\Lambda^\J_k$, then, by commutativity of the upper left triangle, there is nothing to show. Hence, we may restrict to the case where $\Lambda^\J_k$ is admissible and $\Delta^\J$ of dimension $1$, i.e. $\J = [ p \leq p ]$ and $k =0$ or $k=1$.
    Then, however, we may without loss of generality assume that $P = \{p\}$ is a singleton. Since $f \colon Q \to P$ has discrete fibers, this means that $Q$ is discrete. In this case, commutativity of the middle triangle follows immediately from commutativity of the upper left triangle, using path connectedness of $\Delta^\J$.    
\end{proof}
We may now use the refinement functor to obtain model structures which will take care of the pathologies we explain in \cref{rem:so_many_empty_spaces}. The model structure derived from the Joyal-Kan model structure on $\sStrat$ will allow us to think of stratified spaces as fully faithfully embedded into $\infty$-categories (\cref{prop:Quillen_Equ_betw_ref_and_ord}).
We now define model categories presenting homotopy theories of (certain) refined stratified simplicial sets. These are constructed by forcing $\str^{\ared} \to \str$ to be a weak equivalence, and hence turn out to be right Bousfield localizations (and thus coreflective localizations). 
\begin{theorem}\label{prop:ex_red_struct}
    Let $S$ be the class of refinement morphisms $\{ \str^{\ared} \to \str \mid \str \in \sStrat \}$. Then the right Bousfield localization of $\sStratD$ ($\sStratC$) at $S$ exists and is again combinatorial and simplicial. Its defining classes can be characterized as follows: 
    \begin{LocEx}
        \item \label{thm:loc_ex_cof}The cofibrations are generated by the set of stratified boundary inclusions $ \partial \Delta^{[n]} \hookrightarrow \Delta^{[n]}$
        together with the boundary inclusion $\stratBound[1] \hookrightarrow \Delta^1$ into the trivially stratified simplex.
        
        Equivalently, cofibrations are precisely those morphisms $j\colon \str[A] \to \str[B]$ that induce a monomorphism on simplicial sets (i.e. are a cofibration in $\sStratD$ or $\sStratC$) and are furthermore such that the diagram 
        \begin{diagram}
       \str[A]^{\ared} \arrow[r] \arrow[d] & \str[B]^{\ared} \arrow[d] \\
       \str[A] \arrow[r] & \str[B]
        \end{diagram}
         is pushout. In particular, the cofibrant objects are precisely the refined stratified simplicial sets.       
        \item \label{thm:loc_ex_weak} Weak equivalences are precisely those morphisms $f \colon \str \to \str[Y]$ for which $f^{\ared}$ is a weak equivalence in $\sStratD$ ($\sStratC$).
        \item  \label{thm:loc_ex_af} Acyclic fibrations are precisely those morphisms $f \colon \str \to \str[Y]$, for which $f^{\ared}$ is an acyclic fibration in $\sStratD$ (or equivalently in $\sStratC$). 
        In other words, $f$ induces an isomorphism on refined posets and an acyclic fibration on the underlying simplicial sets. 
        \item Fibrations and acyclic cofibrations are the same as in $\sStratD$ ($\sStratC$).
    \end{LocEx}
\end{theorem}
\begin{proof}
    We denote by $I$ the set of generating cofibrations described in \ref{thm:loc_ex_cof} and by $W$ the class of weak equivalences described in \ref{thm:loc_ex_weak}. Furthermore, we denote by $\mathrm{inj}(I)$ the class of morphisms that have the right lifting property with respect to $I$, and denote $\mathrm{cof}(I)$ the class of morphisms that have the left lifting property with respect to $\mathrm{inj}(I)$. Finally, denote by $AC$ and $F$ the classes of acyclic cofibrations and fibrations in $\sStratD$ ($\sStratC$).
    To prove the existence of the localization above, it suffices to show the following:
    \begin{LocExClaim}
        \item \label{thm:loc_ex_claim} $\mathrm{inj}(I)$ is precisely the class of morphisms described in \ref{thm:loc_ex_af}.
        \item \label{thm:loc_ex_claim_2}$\mathrm{cof}(I) \cap W = AC$.
    \end{LocExClaim}
    To see this, note first that it follows from the small object argument that $\mathrm{cof}(I)$ and $\mathrm{inj}(I)$ form a weak factorization system. $AC$ and $F$ form a weak factorization system, by the respective property of $\sStratD$ ($\sStratC$). Hence, it only remains to show $F \cap W = \mathrm{inj}(I)$. That $\mathrm{inj}(I) \subset W$ follows by the characterization in \ref{thm:loc_ex_af}. That $\mathrm{inj}(I) \subset F$ follows from $F = \mathrm{inj}(AC)$ and $AC \subset \mathrm{cof}(I)$. Finally, to see that $F \cap W = \mathrm{inj}(I)$, consider $f \colon \str \to \str[Y] \in F \cap W$ as well as factorization
    \[ \str \xhookrightarrow{i} \hat{\str} \xrightarrow{\hat f} \str[Y]\]
    of $f$ into $i \in \mathrm{cof}(I)$ and $\hat f \in \mathrm{inj}(I)$. Since $f, \hat f \in W$, it follows by two-out-of-three, that the same holds for $i$. It follows from \cref{thm:loc_ex_claim_2} that $i \in AC$. In particular, $i$ has the left lifting property with respect to $f$ from which it follows that $f$ is a retract of $\hat f$, and hence an element of $\mathrm{inj}(I)$.  
    Let us assume that we have shown \cref{thm:loc_ex_claim,thm:loc_ex_claim_2} as well as the equivalence in \ref{thm:loc_ex_cof} for now. Note that the thus defined model category is again combinatorial. Indeed, we have provided a set of generators for cofibrations in \ref{thm:loc_ex_cof} and a set of generators for acyclic cofibrations is given by the ones for $\sStratD$ ($\sStratC$). Next, let us verify simpliciality. Suppose that $i \colon \str[A] \to \str[B]$ lies in $\mathrm{cof}(I)$ and that $j \colon A \to B$ is a cofibration of simplicial sets. We need to show that
    \[
    f \colon \str[C] := \str[B] \otimes A \cup_{\str[A] \otimes A}\str[A] \otimes B \to \str[B] \otimes B =: \str[D]
    \]
    again lies in $\mathrm{cof}(I)$. That the induced map of simplicial sets is a monomorphism is immediate from the corresponding statement on simplicial sets. Note that, by the equivalent characterization of $\mathrm{cof}(I)$ in \ref{thm:loc_ex_cof},
    $\mathrm{cof}(I)$ has the property that for any morphism $g \in \sStratN$ that induces a monomorphism on simplicial sets and
    any $i' \in \mathrm{cof}(I)$ with target the source of $g$, it holds that
    \[
    g \in \mathrm{cof}(I) \iff  g \circ i' \in \mathrm{cof}(I).
    \]
    Thus, it suffices to show that
    \begin{align*}
        \str[B] \otimes A \to \str[B] \otimes B
    \end{align*}
    and 
    \[
    \str[B] \otimes A \to \str[B] \otimes A \cup_{\str[A] \otimes A}\str[A] \otimes B
     \]
    are in $\mathrm{cof}(I)$. Using the stability of $\mathrm{cof}(I)$ under pushouts, we may thus reduce to the cases where either $\str[A]$ or $A$ is empty, i.e. $\str[C]$ is of the form $\str[B] \otimes A$ or $\str[A] \otimes B$.
    Now, again using \ref{thm:loc_ex_cof}, and the fact that pushout diagrams in $\sStratN$ are detected on the poset and simplicial set level, it suffices to show that
    \begin{diagram}\label{diag:thm_loc_ex_c_and_d}
       \rpstr[C] \arrow[r] \arrow[d] & \rpstr[D] \arrow[d] \\
       \pstr[C] \arrow[r] & \pstr[D]
        \end{diagram}
    is a pushout diagram in $\Pos$. 
    Finally, note that applying $-\otimes K$ acts as $- \times \pi_0(A)$ (with $\pi_0(A)$ equipped with the discrete poset structure) both on the level of posets as well as on the level of refined posets. If $\str[A] = \emptyset$, then by assumption $\rpstr[B] \to \pstr[B]$ is an isomorphism and it follows that \cref{diag:thm_loc_ex_c_and_d} is of the form
    \begin{diagram}
       \rpstr[B] \times \pi_0(A) \arrow[r ] \arrow[d, "\cong"] & \rpstr[B] \times \pi_0(B) \arrow[d, "\cong"] \\
       \pstr[B]  \times \pi_0(A)  \arrow[r] & \pstr[B] \times \pi_0(B) \spacecomma
    \end{diagram}
     with horizontals induced by $j$. Since both verticals are isomorphisms, this diagram is pushout. 
    If $A$ is empty, then \cref{diag:thm_loc_ex_c_and_d} is of the form
    \begin{diagram}
       \rpstr[A] \times \pi_0(B) \arrow[r] \arrow[d] & \rpstr[B] \times \pi_0(B) \arrow[d] \\
       \pstr[A] \times \pi_0(B) \arrow[r] & \pstr[B] \times \pi_0(B) \spacecomma
    \end{diagram}
    with horizontal induced by $i$.
    Consequently, it follows from 
    \begin{diagram}
       \rpstr[A]  \arrow[r] \arrow[d] & \rpstr[B] \arrow[d] \\
       \pstr[A]  \arrow[r] & \pstr[B] 
    \end{diagram}
    being pushout by assumption, that \cref{diag:thm_loc_ex_c_and_d} is also pushout in this case. \\
    Finally, if either $i$ or $j$ is an acyclic cofibration, then it follows by the simpliciality of $\sStratD$ ($\sStratC$) and \cref{thm:loc_ex_claim_2} that $f$ is also an acyclic cofibration. \\
    To finish the proof, it remains to show \cref{thm:loc_ex_claim,thm:loc_ex_claim_2} as well as the equivalence in \ref{thm:loc_ex_cof}. This is the content of \cref{lem:char_ac_fib_char_red,lem:char_cofib,lem:char_ac_cof_red}.
    \end{proof}
    \begin{definition}\label{def:ref_mod_struct}
    We denote by $\sStratDR$ and $\sStratCR$, respectively, the simplicial right Bousfield localizations in \cref{prop:ex_red_struct}. They are, respectively, called the \define{diagrammatic} and the \define{categorical} model structure on $\sStrat$.
    Weak equivalences in $\sStratDR$ are called \define{diagrammatic equivalences.}
    Weak equivalences in $\sStratCR$ are called \define{Joyal-Kan equivalences.}
    We call the homotopy theory presented by $\sStratCR$ the $(\infty,1)$-category of refined abstract homotopy types and denote it by $\AbStr^{\ared}$.
\end{definition}
\begin{remark}\label{rem:0-connected_vs_ref_present}
    It follows from \cref{rem:0-connected_vs_ref} that the refined abstract stratified homotopy types are precisely what \cite{Exodromy} calls $0$-connected stratified spaces. In this sense, the part of \cref{prop:ex_red_struct} that is concerned with the Joyal-Kan model structure can be taken to be the construction of a model structure presenting $0$-connected stratified spaces.
\end{remark}
Let us also make the following observation, which is immediate from the characterization of the defining classes in \cref{prop:ex_red_struct}.
\begin{lemma}\label{lem:cof_rep_in_ref}
    For any $\str \in \sStrat$, the natural transformation $\str^{\ared} \to \str$ is an acyclic fibration in $\sStratDR$ ($\sStratCR$). It defines a cofibrant replacement of $\str \in \sStratDR$ ($\sStratCR$).
\end{lemma}
Furthermore, we are going to need the following property of the refined model structures, which follows from \cref{lem:red_pres_filt_colim} and \cref{prop:strat_we_stable_colim}.
\begin{lemma}\label{prop:we_stable_under_colim_ref}
  Weak equivalences in $\sStratDR$ and $\sStratCR$ are stable under filtered colimits.
\end{lemma}
Also note that it follows from \cref{prop:ex_red_struct} together with \cref{prop:c_is_left_bous_of_d} that:
\begin{proposition}\label{prop:c_is_left_bous_of_d_ref}
    $\sStratCR$ is the left Bousfield localization of $\sStratDR$ at the set of stratified inner horn inclusions $\{  \stratHorn \hookrightarrow \stratSim \hookrightarrow \Delta^n \mid 0<k<n\}$.
\end{proposition}
Finally, the following observation will be useful when passing to the topological scenario:
\begin{proposition}\label{prop:char_of_ref_cat_equ}
    Let $f \colon \str \to \str[Y]$ be a stratified simplicial map between stratified simplicial sets $\str, \str[Y]$ whose strata are Kan complexes. Then $f$ is a Joyal-Kan equivalence if and only if the underlying map of simplicial sets $\forget({f})$ is a categorical equivalence (also called Joyal equivalences). 
\end{proposition}
\begin{proof}
  By definition, $f$ is a Joyal-Kan equivalence if and only if $f^{\ared}$ is a categorical equivalence. By \cite[Thm. 0.2.2.2]{haine2018homotopy} this is, in turn, equivalent to the following two conditions being fulfilled.
     \begin{enumerate}
         \item The underlying simplicial map of $f$, $\forget(f)$, is a categorical equivalence.
         \item $f$ induces an isomorphism on refined posets.
     \end{enumerate}
     However, to compute the map on refined posets, by \cref{prop:computing_der_poset}, there is no need to derive at all, and it is given by $P( \forget(f))$. Since $\forget(f)$ is a categorical equivalence, it follows from \cref{lem:pos_of_cat_eq} that the second condition is redundant, as was to be shown.
\end{proof}
We may summarize the whole situation as follows.
\begin{proposition}\label{prop:diag_of_bousfield_loc_sglob}
    The simplicial, combinatorial model structures on $\sStrat$ fit into a diagram of Bousfield localizations
    \begin{diagram}
        \sStratD \arrow[r]\arrow[d] & \sStratC  \arrow[d]\\
        \sStratDR \arrow[r] &\sStratCR
    \end{diagram}
    with the verticals right Bousfield and the horizontals left Bousfield. The verticals are obtained by localizing the stratified inner horn inclusions. The horizontals are obtained by localizing the refinement morphisms $\str^{\ared} \to \str$.
\end{proposition}
Furthermore, consider the following result which - retroactively - justifies the naming conventions for the different notions of equivalences of stratified simplicial sets:
\begin{proposition}\label{prop:names_make_sense}
    Let $f \colon \str \to \str[Y]$ be a stratified simplicial map. Then $f$ is a poset-preserving Joyal-Kan equivalence if and only if $f$ is a Joyal-Kan equivalence and the underlying map of posets, $\pos(f) \colon \pstr \to \pstr[Y] $, is an isomorphism. The analogous result for diagrammatic equivalences holds.
\end{proposition}
\begin{proof}
   Both in the diagrammatic and categorical scenario, the only if case is immediate by \cref{prop:ared_pres_things} together with the idempotency of the refinement functor, and the characterization of weak equivalences in \cref{prop:ex_red_struct}. Next, let us show the if case in the case of Joyal-Kan equivalences. Consider the commutative diagram 
    \begin{diagram}
        \str \arrow[r, "\simeq"] \arrow[d, "\simeq"'] \arrow[ddd, bend right = 50, "\simeq"']& \str[Y] \arrow[d, "\simeq"] \arrow[ddd, bend left = 50, "\simeq"]\\
        \Ex^{\infty} \str \arrow[r] \arrow[d]& \Ex^{\infty} \str[Y] \arrow[d] \\
        \Ex^{\infty} \str^{\ared} \arrow[r] & \Ex^{\infty} \str[Y]^{\ared} \\
        \str^{\ared} \arrow[u, "\simeq "] \arrow[r] & \str[Y]^{\ared} \arrow[u, "\simeq"'] \spacecomma
    \end{diagram}
    and note that all but the upper most row has the property that all stratified simplicial sets involved have strata given by Kan complexes.
    We have marked all maps which are known to be Joyal-Kan equivalences from previous results in this article with a $\simeq$ symbol. That these maps are weak equivalences follows either by assumption, or from the natural map $1 \to \Ex^{\infty} $ even being a poset-preserving diagrammatic equivalence (see \cite[Prop. 3.9]{douteauwaas2021}). A quick diagram chase using the two-out-of-three property shows that all morphisms in the diagram are Joyal-Kan equivalences. It follows from \cref{prop:char_of_ref_cat_equ} that the underlying simplicial map of $\Ex^{\infty} \str \to \Ex^{\infty} \str[Y]$ is a categorical equivalence. Hence, by \cite[Thm. 0.2.2.2]{haine2018homotopy}, using the assumption that $\str \to \str[Y]$ induces an isomorphism on posets, it follows that $\Ex^{\infty} \str \to \Ex^{\infty} \str[Y]$ is a poset-preserving Joyal-Kan equivalence. Finally, the upper two verticals are also poset-preserving Joyal-Kan equivalences (diagrammatic even), from which, again by two-out-of-three, it follows that $\str \to \str[Y]$ is a poset-preserving Joyal-Kan equivalence. It remains to show that a stratified simplicial map that induces isomorphisms on the poset level and is a diagrammatic equivalence is a poset-preserving diagrammatic equivalence. Let $\I$ be a non-degenerate flag of $\pstr$. 
    We obtain an induced commutative diagram of simplicial sets
    \begin{diagram}
        \bigsqcup_{\I' \mapsto \I } \HolIPS[\I'] (\str^{\ared}) \arrow[d] \arrow[r] & \bigsqcup_{\I' \mapsto \pos(f)(\I)} \HolIPS[\I'] (\str[Y]^{\ared}) \arrow[d]\\
        \HolIPS (\str)  \arrow[r] & \HolIPS[\pos(f)(\I)] (\str[Y]) \spacecomma
    \end{diagram}
    where the coproducts are indexed over regular flags mapping to $\I$ (respectively $\pos(f)(\I)$) under $\rpstr \to \pstr$ ($\rpstr[Y] \to \pstr[Y]$). Since $\str \to \str[Y]$ and $\str^{\ared} \to \str[Y]^{\ared}$ are assumed to be injective on the poset level, the two horizontals are well-defined. By assumption, the upper horizontal is a weak homotopy equivalence of simplicial sets. Furthermore, it follows from an application of \cref{prop:refinement_is_ra} that the two verticals are isomorphisms of simplicial sets. Hence, the lower vertical is a weak homotopy equivalence. Since $\str \to \str[Y]$ is assumed to induce an isomorphism on the poset level, it follows that it is a poset-preserving diagrammatic equivalence.
\end{proof}
Now, let us prove the remaining open statements.
\begin{lemma}\label{lem:char_ac_fib_char_red}
    In the framework of \cref{prop:ex_red_struct} and its proof, $\mathrm{inj}(I)$ is the class of stratified maps $f \colon \str \to \str[Y]$ such that $f^{\ared}$ is an isomorphism on posets and the underlying simplicial map of $f$ is a trivial fibration (with respect to any of the model structures on presheaves on $\sSet$). In other words, $f \in \mathrm{inj}(I)$ if and only if $f^{\ared}$ is an acyclic fibration in $\sStratD$ (or equivalently in $\sStratC$).
\end{lemma}    
\begin{proof}
     Denote by $S$ the set of boundary inclusions in $\sSet$ and by $i \in \sStrat $ the remaining cofibration $\stratBound[1] \to \Delta^1$ specified in the statement of the theorem.
    It follows from the adjunction of $\lstr \colon \sSet \to \sStrat$ with the forgetful functor $\sStrat \to \sSet$, that $f \in \mathrm{inj}(\lstr(S))$ if and only if the underlying simplicial map is an acyclic fibration.
    Let us now assume that $f \in \lstr(S)$. Under this assumption, we show that $f \in \mathrm{inj}(i)$ is equivalent to the induced map $\tilde{f} \colon \rpstr \to \rpstr[Y]$ being an isomorphism. 
    $\tilde f$ is an isomorphism, if and only if $\tilde{f}$ is surjective on elements and relations. 
    Assume that $f \in \mathrm{inj}(i \cup \lstr(S))$.
    Surjectivity on elements follows from the fact that the underlying simplicial map of $f$ is surjective (as it is an acyclic fibration). Now, by \cref{prop:explicit_rp}, it suffices to show that any zigzag as in \cref{prop:explicit_rp} lifts. For $1$-simplices that point in direction of $y$, this follows from $f \in \mathrm{inj}(\{ \partial\Delta^{[1]} \to \Delta^{[1]}\})$ . For $1$-simplices pointing in the direction of $x$, this follows from $f \in \mathrm{inj}(i)$. Conversely, let $\tilde f$ be an isomorphism.
    Given a lifting problem as the right square in \begin{diagram}
        \stratBound[1] \arrow[d, hook]\arrow[r, "1"] & \stratBound[1] \arrow[d] \arrow[r] & \str \arrow[d] \\
        \stratSim[1] \arrow[rru, dashed, "g'" {shift right=3}, ]\arrow[r] & \Delta^1 \arrow[r] \arrow[ru, "g"', dashed] & \str[Y],
    \end{diagram}
    by the assumption that $f \in \mathrm{inj}(\lstr(S))$ it follows that a solution $g'$ of the outer rectangle exists. 
    Since $\tilde f$ is an isomorphism, any two points in $\mathcal{X}$ that are mapped into the same stratum of $\mathcal{Y}$ and are connected by a path in the latter, already lie in the same stratum of $\mathcal{X}$. It follows by commutativity of the outer rectangle that $g'$ factors through $\stratBound[1] \to \Delta^1$ into a stratified map $g \colon \Delta^1 \to \str$. Since both left horizontals are epimorphisms, $g$ is a solution for the right lifting square.
\end{proof}
\begin{lemma}\label{lem:char_cofib}
    In the framework of \cref{prop:ex_red_struct} and its proof, $\mathrm{cof}(I)$ is precisely the class of stratified maps $f \colon \str[A] \to \str[B]$ such that the diagram 
   \begin{diagram}
       \str[A]^{\ared} \arrow[r] \arrow[d] & \str[B]^{\ared} \arrow[d] \\
       \str[A] \arrow[r] & \str[B]
   \end{diagram}
    is pushout and such that $f$ is a cofibration in $\sStratD$ (or equivalently $\sStratC$) (i.e. $f$ induces a monomorphism on the simplicial set level).
\end{lemma}    
\begin{proof}
    First, let us show that any $\str[A] \hookrightarrow \str[B]$ in $\mathrm{cof}(I)$ has the pushout property (that it is a monomorphism on simplicial sets is immediate). We only need to show that
     \begin{diagram}\label{diag:pushout_of_posets}
         \rpstr[A] \arrow[r] \arrow[d] &  \rpstr[B] \arrow[d] \\
         \pstr[A] \arrow[r] &\pstr[B]
    \end{diagram}
    is pushout.
    Using the small object argument, we may reduce to showing that this is true for any $j \in I$, and $\str[A] \to \str[B]$ a pushout of $j$. Let us compute explicitly the maps 
    \begin{align*}
        \rpstr[A] &\to  \rpstr[B] \\
        \pstr[A] &\to \pstr[B]
    \end{align*}
    in terms of generators and relations (using the explicit description in \cref{prop:explicit_rp}). If the codomain of $j$ is a simplex of dimension greater than $1$, then both maps on the poset level are isomorphisms. It remains to consider the three cases:
        \begin{align*}
            \emptyset &\hookrightarrow \stratSim[0] ; \\
            \stratBound[1] &\hookrightarrow \stratSim[1];\\
            \stratBound[1] &\hookrightarrow \Delta^1.
        \end{align*}    
    The first of these adds an extra element to $\rpstr[A]$ and $\pstr[A]$. For the second and third case let $p_0, p_1 \in \pstr[A]$ be the strata corresponding to the images of the points in $\stratBound[1]$.
    For the second case, we obtain $\pstr[B]$ by adding precisely one generating relation $r:p_0 \leq p_1$ to $\pstr[A]$. This identifies all elements of $\pstr[A]$ that are now contained in a finite ordered cycle of relations. For the third case, two generating relations $r \colon p_0 \leq p_1$ and $r^{-1} \colon p_1 \leq p_0$ are added to $\pstr[A]$. In both the second and the third case, $\rpstr[B]$ is obtained from $\rpstr[A]$ by adding one additional generating relation $\hat r$, added from $[x_0]$ to $[x_1]$, where $x_0$ and $x_1$ are the respective boundary vertices of the glued in $1$-simplex, and furthermore, one generating relation (pointing in the opposite direction) is added for every $1$-simplex, whose strata become identical in $\pstr$ after adding $r$, ($r,r^{-1}$).
    We may now check by hand that \cref{diag:pushout_of_posets} is pushout. To do this, note that pushouts in partially ordered sets are computed from elements and relations by taking a pushout of the elements in sets and taking the generating relations coming from $\rpstr[B]$ and $\pstr[A]$.
    If none of the three cases above apply, then all horizontals are isomorphisms and there is nothing to be shown.
    In the first case, \cref{diag:pushout_of_posets} is of the form
    \begin{diagram}
         \rpstr[A] \arrow[r] \arrow[d] &  \rpstr[A]  \sqcup [0]\arrow[d] \\
         \pstr[A] \arrow[r] & \pstr[A] \sqcup [0]
    \end{diagram}
    and therefore pushout. In the second and third case, the upper horizontal is surjective on elements. It follows by the explicit construction above that the pushout $ \pstr[A] \cup_{ \rpstr[A]} \cup \rpstr[B] $ may simply be constructed by adding to $\pstr[A]$, all relations of the form $\sstr[A](x) \leq \sstr[A](y)$, where $[y] \leq [x]$ is a generating relation in $\rpstr[B]$, not already present in $\rpstr[A]$. Hence, by our explicit description above, in these cases the pushout is computed by adding the relation $r:  \sstr[A](x_0) \leq \sstr[A](x_1) $, as well as one additional relation $\sstr[A](y_0) \leq \sstr[A](y_1)$, for all edges $y_0 \to y_1$, whose endpoint strata are identified after adding $r$, ($r,r^{-1}$). Note how in both cases the additional relations $\sstr[A](y_0) \leq \sstr[A](y_1)$ are redundant, by their definition. To summarize, we have presented the pushout $\pstr[A] \cup_{ \rpstr[A]} \cup \rpstr[B]$ in terms of the same generators and relations as $\pstr[B]$, which finishes this part of the proof. \\
    Let us now, conversely, show that any map $f \colon \str[A] \to \str[B]$ that induces a monomorphism of the underlying simplicial sets, and a pushout square as in the claim, lies in $\mathrm{cof}(I)$.
    Suppose that we are given a lifting diagram
    \begin{diagram}\label{diag:lift_prob_char_cof}
        \str[A] \arrow[d] \arrow[r] & \str[X] \arrow[d] \\
        \str[B] \arrow[r] & \str[Y] \spacecomma
    \end{diagram}
    with $\str[X] \to \str[Y]$ in $\mathrm{inj}(I)$.
    Note that the induced diagram 
    \begin{diagram}
         \str[A]^{\ared} \arrow[d] \arrow[r] & \str[X]^{\ared} \arrow[d] \\
        \str[B]^{\ared} \arrow[r] & \str[Y]^{\ared} \spaceperiod
    \end{diagram}
    admits a solution. Indeed, the left vertical is a cofibration in $\sStratD$ and, by \cref{lem:char_ac_fib_char_red}, the right vertical is an acyclic fibration $\sStratD$. In particular, we have a solution to the composed diagram
    \begin{diagram}
         \str[A]^{\ared} \arrow[d] \arrow[r] & \str[X]^{\ared} \arrow[d] \arrow[r] &\str[X] \arrow[d] \\
        \str[B]^{\ared} \arrow[r] \arrow[rru, dashed] & \str[Y]^{\ared} \arrow[r] &  \str[Y]  \spaceperiod
    \end{diagram}
   Now, consider the solid commutative diagram
    \begin{diagram}
            \str[A]^{\ared} \arrow[rd] \arrow[dd] \arrow[rr] &  &\str[X] \arrow[dd] \\ 
             &\str[A] \arrow[dd]  \arrow[ru ]& \\
        \str[B]^{\ared} \arrow[rruu, bend left = 30] \arrow[rd] \arrow[rr] &  \arrow[r] &  \str[Y]   \\
        & \str[B] \arrow[ru] \arrow[ruuu, dashed] & \spaceperiod
    \end{diagram}
    A diagram chase shows that the universal property of the pushout induced a dashed solution to our original lifting problem.
\end{proof}
\begin{lemma}\label{lem:char_ac_cof_red}
      In the framework of \cref{prop:ex_red_struct} and its proof, $\mathrm{cof}(I) \cap W$ is precisely the class of acyclic cofibrations in $\sStratD$ ($\sStratC$).
\end{lemma}
\begin{proof}
    Suppose that $j\colon \str[A] \to \str[B]$ is an acyclic cofibration in $\sStratD$ ($\sStratC$). By \cref{prop:ared_pres_things}, it follows that $j^{\ared}$ is an isomorphism on posets. Consequently, the diagram
   \begin{diagram}
       \str[A]^{\ared} \arrow[r] \arrow[d] & \str[B]^{\ared} \arrow[d] \\
       \str[A] \arrow[r] & \str[B]
   \end{diagram}
   is pushout, which by \cref{lem:char_cofib} implies that $j$ lies in $\mathrm{cof}(I)$. Furthermore, again by \cref{prop:ared_pres_things}, we also have that $j^{\ared}$ is a weak equivalence, i.e., that $j \in W$. Now, conversely, assume that $j \in \mathrm{cof}(I) \cap W$. By the definition of $W$, it follows that $j^{\ared}$ is a weak equivalence in $\sStratD$ ($\sStratC$). As $j$ is given by a monomorphism on simplicial sets, it thus follows that $j^{\ared}$ is an acyclic cofibration in $\sStratD$ ($\sStratC$). Since, by \cref{lem:char_cofib}, the diagram 
    \begin{diagram}
       \str[A]^{\ared} \arrow[r] \arrow[d] & \str[B]^{\ared} \arrow[d] \\
       \str[A] \arrow[r] & \str[B]
   \end{diagram}
   is pushout, it follows that $j$ is also an acyclic cofibration in $\sStratD$ ($\sStratC$).
\end{proof}
\subsection{Refined abstract stratified homotopy types and layered \texorpdfstring{$\infty$}{infinity}-categories}\label{subsec:from_ref_to_ordered_qc}
Let us give an alternative description of the homotopy theory defined by categorical model structure on $\sStrat$. It turns out that it is a fully faithful subcategory of the infinity category of all small infinity categories $\iCat$.
\begin{definition}\cite{Exodromy}
Let $X \in \sSetN$ be a quasi-category. We say $X$ is \define{layered}, if the natural functor
\[
X \to P(X)
\]
is conservative. More generally, we say that an arbitrary $Y \in \sSetN$ is \define{layered}, if this holds for any fibrant replacement of $Y$ in the Joyal model structure, $\sSetJoy$. 
\end{definition}
\begin{remark}\label{rem:equ_desc_of_ordered_icat}
    In other words, a quasi-category $X \in \sSetN$ is layered if and only if each endomorphism in $X$ is an isomorphism. This has the effect that the isomorphism classes naturally carry the structure of a poset, with a relation $[x] \leq [y]$ if and only if there is a morphism $x \to y$. This poset then agrees with $P(X)$. 
\end{remark}
\begin{notation}
    We denote by $\iCatO$ the full subcategory of $\iCat$ given by the layered quasi-categories.
\end{notation}
Let us construct a model structure on $\sSetN$ corresponding to $\iCatO$.
\begin{construction}
    Denote by $E$ the quotient of $\Delta^2$ obtained by collapsing the edge $[0,2]$ and identifying the vertices $[0]$ and $[1]$. 
    In other words, we have specified the generating data for a free endomorphism, which has a right inverse. Furthermore, denote by $S^1$ the quotient of $\Delta^{1}$ by $\partial \Delta^1$. The inclusion $\Delta^1 \hookrightarrow \Delta^2$, mapping to the $[0,1]$ face, induces an inclusion of simplicial sets
                \[
                l \colon S^1 \hookrightarrow E.
                \]
    We denote by $\sSetOrd$ the left Bousfield localization of $\sSetJoy$ at $l$, which exists by \cite[Thm. 4.7]{BarwickLeftRight}. 
\end{construction}  
\begin{proposition}\label{prop:sso_pres_ocat}
    $\sSetOrd$ is a model for the $\infty$-category $\iCatO$. 
\end{proposition}
\begin{proof}
    Since, $\iCatO$ is a full subcategory of $\iCat$ and $\sSetOrd$ is a left Bousfield localization of $\sSetJoy$, which models $\iCat$, we only need to show that the fibrant objects of $\sSetOrd$ are precisely the layered quasi-categories. Now, note that a quasi-category $X$ is $l$-local, if and only if the induced simplicial map
    \[
    \sSet (E,X)^{\simeq} \to \sSet(S^1,X)^{\simeq},
    \]
    where $(-)^\simeq$ denotes the maximal Kan complex contained in these quasi categories, 
    is a weak homotopy equivalence (indeed this follows from the fact that the latter Kan complexes define derived mapping spaces for $\sSetJoy$). The path components of the left-hand side correspond to (isomorphism classes of)  morphisms which have a right inverse. The path components on the right-hand side correspond to (isomorphism classes of) endomorphisms. Hence, this map being a weak equivalence implies that every endomorphism in $X$ has a right inverse. Since this also holds for the respective right inverses, it follows that every endomorphism in $X$ is an isomorphism, i.e. that $X$ is layered (\cref{rem:equ_desc_of_ordered_icat}). Conversely, if every endomorphism of $X$ is an isomorphism, then every simplicial map from $A=E,S^1$ to $X$ has image in $X^{\simeq}$. Hence, it follows (by \cite[Cor. 3.5.12.]{HigherCatCisinki} ) that \[\sSet (A,X)^{\simeq} = \sSet(A, X^{\simeq})^{\simeq} = \sSet(A, X^\simeq),\] as the middle term is already a Kan complex. It is not hard to see that $S^1 \hookrightarrow E$ is a weak homotopy equivalence of simplicial sets (see the proof of \cref{prop:Quillen_Equ_betw_ref_and_ord} below), which implies that 
    \[
    \sSet (E,X)^{\simeq} = \sSet (E,X) \to \sSet(S^1,X) = \sSet(S^1,X)^{\simeq},
    \]
    is also a weak homotopy equivalence.
\end{proof}
We may now expose $\sStratCR$ as a different model for the homotopy theory of layered $\infty$-categories. The $\infty$-categorical version of this statement was already proven in \cite[\nopp 2.3.8]{Exodromy}. Here is the model-categorical version of this statement:
\begin{theorem}\label{prop:Quillen_Equ_betw_ref_and_ord}
    The adjunction 
        \begin{align*}
          \lstr \colon \sSetN  & \rightleftharpoons  \sStratN \colon \forget
        \end{align*}
    induces a Quillen equivalence between $\sSetOrd$ and $\sStratCRN$.
\end{theorem}   
\begin{proof}
    We begin by showing that $\lstr$ is left Quillen. It follows immediately from the construction of $\sStratCRN$ that $\lstr$ is left Quillen as a functor with domain $\sSetJoy$. For cofibrations, this follows by construction. For acyclic cofibrations, note that any categorical equivalence $X \to Y$ induces a morphism $\lstr (X) \to \lstr(Y)$ that is an isomorphism on posets $P(X) \to  P(Y)$. In particular, by definition of the model structure on $\sStratCN$, $\lstr (X) \to \lstr(Y)$ is a weak equivalence. 
    Now, consider $\lstr(l) \in \sStratN$. On the poset level, $\lstr(l)$ is given by the identity on the poset $[0]$. It follows, by construction of $\sStratCN$ that $\lstr(l)$ is a weak equivalence if and only if it is a weak homotopy equivalence of (trivially stratified) simplicial sets. Indeed, $l$ is given by the pushout 
    \begin{diagram}
        \Delta^1 \arrow[r] \arrow[d] & (\Delta^2)/\Delta^{[0,2]} \arrow[d] \\
        S^1 \arrow[r] & E.
    \end{diagram}
    The upper horizontal is an acyclic cofibration in the Kan-Quillen model structure on simplicial sets, hence so is $l$. By the universal property of the left Bousfield localization, it follows that $\lstr \colon \sSetOrd \to \sStratN$ is indeed left Quillen. It remains to show that $\lstr \dashv \forget$ is a Quillen equivalence. Let $\str[Y] \in \sStratCRN$ be fibrant (i.e. $\sstr[Y] \colon \ustr[Y] \to \pstr[Y]$ a conservative functor of quasi-categories) and $X \in \sSetN$. Consider a morphism 
    \[
    f \colon \lstr (X) \to \str[Y]
    \]
    and its adjoint
    \[
    g \colon X \to \ustr[Y].
    \]
    We need to show that $f$ is a weak equivalence, if and only if $g$ is a weak equivalence. 
    By replacing $X$ and $\str[Y]$ fibrantly and cofibrantly, respectively, we may without loss of generality assume that $X$ and $\str[Y]$ are bifibrant.
    Note that by definition of the model structure on $\sSetOrd$, $\lstr$ sends fibrant objects in $\sSetOrd$ (i.e. layered infinity categories) to fibrant objects in $\sStratCRN$. Similarly, as every object in $\sSetOrd$ is cofibrant, $\forget$ preserves cofibrant objects. It follows by the construction of $\sStratCRN$ as a right Bousfield localization (and the Whitehead theorem for Bousfield localizations \cite[Thm. 3.2.13]{hirschhornModel}), that $f$ is a weak equivalence in $\sStratCRN$, if and only if it is a weak equivalence in $\sStratCN$. Hence, we may without loss of generality, assume that $f$ is the identity on posets, for $P= \pstr[Y]$. Thus, from \cref{recol:haine_mod_cat} it follows that $f$ is a weak equivalence, if and only if the underlying map of simplicial sets (which is $g$) is a categorical equivalence. Finally, again using the local Whitehead theorem, $g$ is a categorical equivalence if and only if it is a weak equivalence in $\sSetOrd$.
\end{proof}
\begin{remark}
    Denote by $\AbStr^{\ared}$ the full coreflective subcategory of $\AbStr$ given by refined abstract stratified homotopy types. Then it follows from \cref{prop:Quillen_Equ_betw_ref_and_ord} that the forgetful functor $\sStratN \to \sSetN$ induces a fully faithful reflective embedding
    \[
    \AbStr^{\ared} \hookrightarrow \iCat
    \]
    with essential image the full subcategory of layered $\infty$-categories $\iCatO$.
\end{remark}
\subsection{Homotopy links for the global stratified setting}
Weak equivalences in $\sSetPD$ can be detected entirely in terms of generalized homotopy links. The question arises whether we can make a similar argument in the global scenario. To do so, we first need a global version of the homotopy link.
\begin{definition}\label{def:ext_homotopy_link}
    For $n \in \mathbb N$ and $\str \in \sStratN$, we denote 
       \[  \AltHolIPS(\str) := \sStrat(\Delta^{[n]}, \str) \]
    and call this simplicial set the $n$-th extended homotopy link of $\str$.
\end{definition}
\begin{remark}
    Note that we may decompose the extended homotopy link into two parts:
    \[
    \AltHolIPS(\str)= \bigsqcup_{\I \in (NP)_n, \I \textnormal{ n.d.} } \HolIPS (\str) \sqcup \bigsqcup_{\J \in (\nerve P)_n, \J \textnormal{ d. }} \sSetP[{\pstr}]( \Delta^\J, \str),
    \]
    where the left-hand union ranges over regular flags, and the right-hand union over degenerate flags.
    Now if $\J$ degenerates from a regular flag $\I$ of $\pstr$, then 
    $\Delta^\J$ and $\Delta^\I$ are stratum-preserving homotopy equivalent. It follows that 
    $\sSetP[{\pstr}]( \Delta^\J, \str)$ is naturally homotopy equivalent to $\sSetP[{\pstr}](\Delta^\I,\str) = \HolIPS (\str)$. In other words, $\AltHolIPS$ carries a lot of homotopy-theoretically redundant data, which is already contained in links of lower dimension. This extra data is only necessary to make $\AltHolIPS$ functorial in morphisms that do not induce injections on the poset level.
\end{remark}
Extended homotopy links turn out to create weak equivalences in $\sStratDR$. To see this, note first that:
\begin{proposition}\label{prop:hol_detect_actual_poset}
    A stratified simplicial map $f \colon \str \to \str[Y] \in \sStratN$ that induces isomorphisms on $\pi_0 \AltHolIPS[0]$ and $\pi_0 \AltHolIPS[1]$ induces an isomorphism $\rpstr \to \rpstr[Y]$.
\end{proposition}
\begin{proof}
    Consider the explicit construction of $\rpstr$ in terms of elements and relations in \cref{prop:explicit_rp}. The elements correspond precisely to the elements of $\pi_0 \AltHolIPS[0] (\str)$. The generating relations correspond precisely to the elements of $\pi_0 \AltHolIPS[1] (\str)$ (with components of degenerate flags corresponding to equalities). Hence, the result follows.
\end{proof}
We may then show:
\begin{proposition}\label{prop:equ_char_of_ref_diag_equ}
    A stratified simplicial map $f \colon \str \to \str[Y]$ is a diagrammatic equivalence if and only if it induces weak homotopy equivalences on all extended homotopy links $\AltHolIPS[n]$, for $n \geq 0$.
\end{proposition}
\begin{proof}
   It follows from \cref{prop:refinement_is_ra}, that $\AltHolIPS[n]$ sends the refinement morphisms $\str^{\ared} \to \str$ into isomorphisms. Consequently, by the characterization of refined diagrammatic equivalences in \cref{prop:ex_red_struct}, we may, without loss of generality, assume that $\str$ and $\str[Y]$ are refined. By \cref{prop:hol_detect_actual_poset}, under both assumptions that we want to show are equivalent the induced map $\pstr= \rpstr \to \rpstr[Y] = \pstr[Y]$ is an isomorphism. Hence, we may, without loss of generality, assume that it is given by the identity. Since $\str$ and $\str[Y]$ are colocal objects with respect to the right Bousfield localization defining $\sStratDR$, $f$ is a diagrammatic equivalence if and only if it is a poset-preserving diagrammatic equivalence. In particular, this is the case if and only if $f$ induces weak equivalences on all homotopy links. Since the extended homotopy links are given as coproducts of all homotopy links and spaces naturally weakly equivalent to the latter, it follows that $f$ is poset-preserving diagrammatic equivalence if and only if it induces an equivalence on extended homotopy links.
\end{proof}
\begin{remark}
    We may rephrase \cref{prop:equ_char_of_ref_diag_equ} in the sense that the weak equivalences on $\sStratDR$ are transported from weak equivalences of bisimplicial sets (interpreted as simplicial presheaves on $\Delta$), under the functor into bisimplicial sets
    \begin{align*}
    \AltHolIPS[] \colon \sStratN \to \bisSet 
    \end{align*}
    induced by the functoriality of $\AltHolIPS$ in $n$. This justifies the name diagrammatic equivalences, as these equivalences are created by the diagram of extended homotopy links. It seems plausible that $\sStratDR$ is Quillen equivalent to a localization of $\bisSet$ equipped with the Reedy model structure. This is also supported by the fact that $\sStratCR$ is a left Bousfield localization of $\sStratDR$, which is equivalent to $\sSetOrd$, whose homotopy theory may in turn be presented as a left Bousfield localization of complete Segal spaces.
\end{remark}

\subsection{Stratified mapping spaces}\label{subsec:strat_mapping_spaces}
Give two layered $\infty$-categories $X$ and $Y$, the $\infty$-category of functors $Y^X$ is itself layered. Indeed, it follows from the fact that isomorphisms of functors are detected pointwise that it even suffices for $Y$ to be layered. Similarly, in the world of topological stratified spaces (more specifically homotopically stratified spaces) \cite{HughesPathSpaces} equipped the space of stratified maps with a natural decomposition (which generally may not be a stratification) and investigated the lifting properties of such mapping spaces. In \cite{nand2019simplicial}, the author refined the topology on these mapping spaces in order to obtain internal mapping spaces, at least for stratified spaces with non-empty strata. Hence, it is not surprising that the homotopy theories on stratified simplicial sets defined in this paper admit a notion of stratified mapping space. In other words, in this subsection we prove that all of the model structures on $\sStratN$ presented in this paper are cartesian closed (\cref{thm:cartesian_closure}). Let us begin with the corresponding statement on $1$-categories.
\begin{proposition}\label{prop:cartesian_closed}
    $\sStratN$ is a cartesian closed category. 
\end{proposition}
\begin{proof}
    Recall first that the category $\Pos$ is also cartesian closed. Given two posets $\pos, \pos[Q]$, the inner hom $\pos^{\pos[Q]}$ is obtained by equipping $\Pos(\pos, \pos[Q])$ with the poset structure given by 
    \[
    f \leq g \iff \forall p \in \pos \colon f(p) \leq g(p),
    \]
    for $f,g \in \Pos(\pos, \pos[Q])$. Next, note that the adjunction $\pos (-) \dashv \nerve (-)$ between simplicial sets and posets has the property that the left adjoint preserves finite products. It follows by an easy application of the Yoneda lemma that there is a natural isomorphism
    \[
    \nerve(\pos)^{\nerve(\pos[Q])} \to \nerve( \pos^{\pos[Q]}) .
    \]
    On vertices, it is simply given by the identification $\sSetN( \nerve (\pos), \nerve(\pos[Q])) \cong \Pos ( \pos, \pos [Q])$, which under the adjunction $\pos \dashv \nerve$ entirely describes the map. Now, let $\str,\str[Y] \in \sStratN$. We construct the exponential object $\tstr[Y]^{\tstr}$, that is, we construct a stratified simplicial set $\tstr[Y]^{\tstr}$ together with a natural isomorphism $\sStratN( - \times \str, \str[Y]) \cong \sStratN( -,\tstr[Y]^{\tstr})$.
    Consider the pullback diagram of simplicial sets
    \begin{diagram}\label{diag:construction_inner_stratified_hom}
        \ustr[Y]^{\ustr} \times_{ \nerve(\pstr[Y])^{\ustr}} \nerve(\pstr[Y])^{ \nerve(\pstr)} \arrow[r] \arrow[d] & \arrow[d]\ustr[Y]^{\ustr} \\
       \nerve( \pstr[Y]^{\pstr}) \cong \nerve(\pstr[Y])^{ \nerve(\pstr)} \arrow[r] & \nerve(\pstr[Y])^{\ustr}.
    \end{diagram}
    The left-hand side defines a stratified simplicial set over the poset $\pstr[Y]^{\pstr}$.
    Let us denote the latter by $\tstr[Y]^{\tstr}$.
    A morphism from a stratified simplicial set $\tstr[Z]$ into this stratified simplicial set corresponds to the data of a morphism $\nerve(\pstr[Z]) \to \nerve(\pstr[Y]^{\pstr})$, together with a morphism $\utstr[Z] \to \utstr[Y]^{\utstr[X]}$ making the induced diagram
        \begin{diagram}
       \ustr[Z] \arrow[rr] \arrow[d] & & \ustr[Y]^{\ustr} \arrow[d] \\
       \nerve(\pstr[Z]) \arrow[r]&\nerve( \pstr[Y]^{\pstr}) \cong \nerve(\pstr[Y])^{ \nerve(\pstr)} \arrow[r] & \nerve(\pstr[Y])^{\ustr}.
    \end{diagram}
    commute. Using the cartesian structure of $\sSetN$, this in turn specifies the same data as a commutative diagram
    \begin{diagram}{}
        \ustr[Z] \times \ustr \arrow[d, "{\sstr[Z] \times \sstr}"] \arrow[r] & \ustr[Y] \arrow[d, "{\sstr[Y]}"] \\
        \nerve( \pstr[Z] \times \pstr)\cong \nerve( \pstr[Z]) \times \nerve (\pstr[X]) \arrow[r] & \pstr[Y]
    \end{diagram}
    that is of a morphism $\str[Z] \times \str \to \str[Y]$. The naturality of the thus constructed bijection
    \[
    \sStratN (\str[Z] \times \str, \str[Y]) \cong \sStratN(\str[Z], \tstr[Y]^{\tstr})
    \]
    shows that $\tstr[Y]^{\tstr}$ defines the required exponential object. 
\end{proof}
\begin{lemma}\label{lem:product_pres_equ}
     The functor $- \times - \colon \sStratN \to \sStratN$ preserves (poset-preserving) diagrammatic and (poset-preserving) Joyal-Kan equivalences.  
\end{lemma}
\begin{proof}
    Clearly, if $f \colon \str \to \str[Y]$ induces an isomorphism on posets, then so does every product $f \times 1_{\str[Z]}$, for $\str[Z] \in \sStratN$. Thus, it follows from \cref{prop:names_make_sense} that it suffices to show the diagrammatic and the Joyal-Kan case, and the poset-preserving versions follow from the latter. 
    The case of diagrammatic equivalences is immediate from \cref{prop:equ_char_of_ref_diag_equ}, which states that the functor $\AltHolIPS[] \colon \bisSet$ creates weak equivalences and the fact that the latter commutes with products. Finally, for the Joyal-Kan case, suppose that $f \colon \str \to \str[Y]$ is a Joyal-Kan equivalence and consider the following induced commutative diagram:
    \begin{diagram}
        \str \times \str[Z] \arrow[r] \arrow[d] & \Ex^{\infty}\str \times  \Ex^{\infty} \str[Z] \arrow[d]    \\ 
        \str[Y] \times \str[Z] \arrow[r] &    \Ex^{\infty}\str[Y] \times  \Ex^{\infty} \str[Z]  \spaceperiod
    \end{diagram}
    Since the natural transformation $1 \to \Ex^{\infty}$ is a poset-preserving diagrammatic equivalence, it follows from the diagrammatic case that both horizontals are diagrammatic and thus also Joyal-Kan equivalences. Hence, by two-out-of-three, we only need to show that the right vertical is a Joyal-Kan equivalence.
    By \cref{prop:char_of_ref_cat_equ}, using the fact that a product of Kan complexes is a Kan complex, it follows that the right vertical is a Joyal-Kan equivalence if and only if the underlying simplicial map is a categorical equivalence. This map is given by the product of the underlying simplicial maps of $1_{\Ex^{\infty}\str[Z]}$ and $\Ex^{\infty}(f)$. By assumption, and since $\Ex^{\infty}$ preserves Joyal-Kan equivalences, $\Ex^{\infty}(f)$ is a Joyal-Kan equivalence. We may again apply \cref{prop:char_of_ref_cat_equ}, from which it follows that the underlying simplicial map of $\Ex^{\infty}(f)$ is a categorical equivalence. Thus, the claim follows from the fact that products in $\sSetN$ preserve categorical equivalences.
\end{proof}
\begin{lemma}\label{lem:suff_to_pres_prod}
    Given a model category $\cat[M]$, suppose that the product functor $- \times - \colon \cat[M] \times \cat[M] \to \cat[M]$ preserves colimits and weak equivalences in both arguments, and also is such that for any pair of cofibrations $i \colon A \to B$ and $j\colon A' \to B'$ the induced morphism 
    \[
    i \boxtimes j \colon A \times B' \cup_{A \times A'} B \times A' \to  B \times B'
    \]
    is a cofibration. Then $- \times -$ is a Quillen bifunctor.
\end{lemma}
\begin{proof}
    We need to show that, given two cofibrations as in the statement of the lemma, if (without loss of generality) $i$ is additionally a weak equivalence, then so is $i \boxtimes j$. Consider the diagram 
    \begin{diagram}
        A \times A' \arrow[d] \arrow[r]& B \times A' \arrow[d] & \\
        A \times B' \arrow[r] & A \times B' \cup_{A \times A'} B \times A' \arrow[r] & B \times B' .
    \end{diagram}
    By assumption, the upper horizontal and lower horizontal compositions are weak equivalences. Since, in addition to this, the upper horizontal is a cofibration and the square is pushout, it follows that its parallel is also a weak equivalence. Hence, by two out of three, so is the right lower horizontal.
\end{proof}
\begin{theorem}\label{thm:cartesian_closure}
    Let $\sStratN$ be equipped with any of the model structures of \cref{subsec:from_local_to_global,subsec:refining_strat_sset}. Then $\sStratN$ is a cartesian closed model category.
\end{theorem}
\begin{proof}
    We need to show that the map from the initial object $\emptyset \in \sStratN$ to the terminal objects $\star \in \sStratN$ is a cofibration and, furthermore, that $- \times - \colon \sStratN \times \sStratN \to \sStratN$ is a Quillen bifunctor. 
    The former statement holds, since $(\emptyset \to \star) \cong (\partial \Delta^{[0]} \hookrightarrow \Delta^{[0]})$, which is a generator for the cofibrations in any of the model structures (by \cref{cor:cof_gen_sStratD,prop:ex_red_struct}).
    For the second statement, we make use of \cref{prop:cartesian_closed,lem:suff_to_pres_prod,lem:product_pres_equ} and it remains to show that for every pair of cofibrations $i \colon \str[A] \hookrightarrow \str[B]$, $j \colon \str[A]' \hookrightarrow \str[B]'$ the induced stratified simplicial map \[
    i \boxtimes j \colon \str[A] \times \str[B]' \cup_{\str[A] \times \str[A]'} \str[B] \times \str[A]' \to \str[B] \times \str[B']\] is a cofibration.
    Both in $\sStratCN$ and in $\sStratDN$ a map is a cofibration, if and only if the underlying simplicial map is a cofibration (i.e. a monomorphism) hence in these cases the $i\boxtimes j$ is a cofibration, since the underlying simplicial map $\forget (i\boxtimes j) \cong \forget(i) \boxtimes \forget(j)$ is a cofibration.
    Even more, by \cite[Cor. 4.2.5]{hovey2007model} we only need to consider the cases where $i$ and $j$ are generating cofibrations. For the cases $\sStratCRN$ and $\sStratDRN$, it follows that we only need to show that 
    \[
    \partial \Delta^{[n]} \times \Delta^{[m]} \cup_{\partial \Delta^{[n]} \times \partial \Delta^{[m]}}  \Delta^{[n]} \times \partial \Delta^{[m]} \to \Delta^{[n]} \times \Delta^{[m]}
    \]
    is a cofibration. Since both the source and target of this cofibration are refined, it follows from the cases $\sStratCN$ and $\sStratDN$, together with the characterization of cofibrations between cofibrant objects in a right-bousfield localization (\cite[3.3.16]{hirschhornModel}) that $i \boxtimes j$ is a cofibration. 
\end{proof}
One of the main results of \cite{HughesPathSpaces} was that for certain particularly convenient stratified spaces $\tstr$ - so-called homotopically stratified spaces - for any closed union of strata $\tstr[A] \hookrightarrow \tstr$ the starting point evaluation map from the space of stratified paths $\sReal{\Delta^{[1]}} \to \tstr$ starting in $\tstr[A]$, $\textnormal{Path}_{nsp}(\tstr[A], \tstr[X])$, is a stratified fibration (i.e., has the right lifting property with respect to inclusions into the stratified cylinder).
Homotopically stratified spaces have the property that they are mapped into fibrant objects in $\sStratCN$, and being a fibration in $\sStratCN$ is even a stronger property than just lifting (simplicial) stratified homotopies. Thus, we may interpret the following result as a combinatorial analogue of \cite[Main Result]{HughesPathSpaces}.
\begin{construction}\label{con:global_exit_paths}
    Let $\sStratN$ be equipped with one of the model structures of \cref{subsec:refining_strat_sset}.
    Let $\str \in \sStratN$ and let $A \subset \ustr$ be a simplicial subset. Denote by $\str[A]$ the stratified simplicial set $\ustr[A] \to \ustr[X] \to \pstr$. Now, consider the pullback diagram in $\StratN$
        \begin{diagram}
            \textnormal{Path}_{nsp}(\str, \str[A]) :=\str^{\Delta^{[1]}}\times_{\ev_0} \str[A] \arrow[d]    \arrow[r] & \str^{\Delta^{[1]}} \arrow[d, "\ev_0"]\\
            \str[A] \arrow[r] & \str .
        \end{diagram}
     $\textnormal{Path}_{nsp}(\str, \str[A])$ is a stratified space over $\{ (p,q) \in \pstr \times \pstr \mid p \leq q\}$. Its vertices are precisely the $1$-simplices (i.e. paths) in $\str$, starting in $\str[A]$.
     Now, if $\str$ is fibrant, then it follows from the cartesian closedness of $\StratN$ that the right-hand vertical is a fibration. Consequently, so is the starting point evaluation map $\textnormal{Path}_{nsp}(\str, \str[A])  \to \str[A]$. In particular, this map has the right lifting property with respect to the acyclic cofibrations $\str[B] \hookrightarrow \str[B] \otimes \Delta^{1}$.
\end{construction}

\section*{Acknowledgments}
The author is supported by a PhD-stipend of the Landesgraduiertenförderung Baden-Württemberg. 
\printbibliography

\end{document}